\documentclass[envcountsect,referee]{svjour3}
\smartqed  
\usepackage{graphicx}
\usepackage{latexsym,bm,amsmath,amssymb,mathrsfs}
\usepackage{rotating}
\usepackage{multirow,booktabs,cases}
\usepackage[misc]{ifsym}
\usepackage[style=1]{mdframed}
\usepackage{algorithm,algorithmic}
\usepackage[colorlinks=true]{hyperref}
\hypersetup{urlcolor=blue,citecolor=blue,linkcolor=blue}
\usepackage{epstopdf}

\def\[{\begin{equation}}
\def\]{\end{equation}}

\def\nn{\nonumber}
\def\t{\top}
\def\A{{\mathcal A}}

\numberwithin{equation}{section}
\allowdisplaybreaks

\begin{document}
\graphicspath{{./PIC/}}

\title{Error analysis of approximation algorithm for standard bi-quadratic programming}


\author{Chen Ling \and Hongjin He \and Liqun Qi}


\institute{C. Ling \and H. He \at
Department of Mathematics, School of Science, Hangzhou Dianzi University, Hangzhou, 310018, China.\\
\email{cling\_zufe@sina.com}
\and H. He \at
\email{hehjmath@hdu.edu.cn}
\and L. Qi \at
Department of Applied Mathematics, The Hong Kong Polytechnic University, Hung Hom, Kowloon, Hong Kong. \\
\email{maqilq@polyu.edu.hk}
 }

\date{Received: date / Accepted: date}

\maketitle

\begin{abstract}
We consider the problem of approximately solving a {\it standard bi-quadratic programming} (StBQP), which is NP-hard. After reformulating the original  problem as an equivalent copositive tensor programming, we show how to approximate the optimal solution by approximating the cone of copositive tensors via a serial polyhedral cones. The established quality of approximation shows that, a {\it polynomial time approximation scheme} (PTAS) for solving StBQP exists and can be extended to solving standard multi-quadratic programming. Some numerical examples are provided to illustrate our approach.



\end{abstract}

\section{Introduction}\label{Sec1}
We consider a polynomial optimization problem of the form
\begin{equation}\label{polyOpt}
\begin{array}{rl}
p^{\rm min}_{\mathcal{A}}={\rm min}&p_{\mathcal{A}}(x,y):=\displaystyle\sum_{i,j=1}^n\sum_{k,l=1}^ma_{ijkl}x_{i}x_{j}y_{k}y_l\\
{\rm ~~~~s.t.}&x\in \Delta_n,~~y\in \Delta_m,
\end{array}
\end{equation}
where $$
\Delta_s:=\left\{z\in \Re_+^s~\Big{|}~\sum_{i=1}^sz_i=1\right\}
$$
is the standard simplex and $\Re_+^s$ denotes the non-negative orthant in $s$-dimensional Euclidean space $\Re^s$. Here, $\mathcal{A}:= (a_{ijkl})_{1\leq i,j\leq n,1\leq k,l\leq m}$ is a real $(2,2)$-th order $n\times n\times m\times m$-dimensional tensor. Without loss of generality, we assume that the tensor $\mathcal{A}$ satisfies the following symmetry condition:
\begin{equation}\label{partsymm}
a_{ijkl}=a_{jikl}=a_{jilk}, ~~\forall~i,j=1,2,\ldots,n;~k,l=1,2,\ldots,m.
\end{equation}
We call the tensor satisfying (\ref{partsymm}) is partially symmetric. It is easy to see that, in case where all $a_{ijkl}$ are independent of the indices $k$ and $l$, i.e., $a_{ijkl} = b_{ij}$ for every
$i,j=1,\ldots,n$, the original problem (\ref{polyOpt}) reduces to the following {\it standard quadratic programming} (StQP)
\begin{equation}\label{StQPmod}
\begin{array}{rl}
f_{\rm min}={\rm min}&f(x):=\displaystyle\sum_{i,j=1}^nb_{ij}x_{i}x_{j},\\
{\rm s.t.}&x\in \Delta_n.
\end{array}
\end{equation}
Hence, the problem (\ref{polyOpt}) is called a {\it standard bi-quadratic programming} (StBQP). StQP not only occur frequently as subproblem in escape procedures for general quadratic programming, but also have manifold applications, e.g., in portfolio selection and in maximum weight clique problem
for undirected graphs. For details, see, e.g. \cite{BBPP99,BP05,M52,P80} and references therein. If we consider portfolio selection problems with two groups of securities whose investment decisions influence each other, then a generalized mean-variance model can be expressed as a StBQP, see \cite{BLQZ12} for details. In that paper, some optimality conditions of  StBQP were studied, and based upon a continuously differentiable penalty function, the original problem was converted into the problem of locating an unconstrained global minimizer of bi-quartic problem.

In terms of $\mathcal{A}$, it is easy to see that the objective function in (\ref{polyOpt}) can be written briefly as
$$
p_{\mathcal{A}}(x,y)=(\mathcal{A}xx^\top)\bullet(yy^\top)=(yy^\top\mathcal{A})\bullet(xx^\top),
$$
where $$\mathcal{A}xx^\top=\left(\underset{i,j=1}{\sum^n}a_{ijkl}x_ix_j\right)_{1\leq k,l\leq m}~~{\rm and}~~yy^\top\mathcal{A}=\left(\sum_{k,l=1}^ma_{ijkl}y_ky_l\right)_{1\leq i,j\leq n}$$ are two $m\times m$ and $n\times n$ symmetric matrices, respectively, and $X\bullet Y$ stands for usual Frobenius inner product for matrices, i.e., $X\bullet Y = {\rm tr}(X^\top Y )$.

The problem of solving (\ref{polyOpt}) is NP-hard, even if the objective $p$ is a quadratic function, see \cite{BK02,Nes99,NWY00}. Therefore,
designing some efficient algorithms for finding approximation solutions of (\ref{polyOpt}) is of interest. In \cite{LZQ14}, some approximation bounds for the standard
bi-quadratic optimization problem were presented. Moreover, by using the variables $z_i^2$ and $w_j^2$ to replace $x_i$ and $y_j$ respectively, the original problem can be rewritten as
\begin{equation*}
\begin{array}{rl}
\min&g(z,w):=\displaystyle\sum_{i,j=1}^n\sum_{k,l=1}^ma_{ijkl}z^2_{i}z^2_{j}w^2_{k}w^2_l\\
{\rm ~~~~s.t.}&\|z\|^2=1,\;\;\|w\|^2=1,\;\;(z,w)\in\Re^n\times\Re^m.
\end{array}
\end{equation*}
Base on this, a polynomial-time approximation algorithm with relative approximation ratio was studied, the obtained result is a bi-quadratic version of that presented in \cite{LHZ12,So11}.
It is well-known that StQP does allow a {\it polynomial time approximation scheme} (PTAS), as was shown by Bomze and De Klerk \cite{BK02}. For the more general minimization of polynomial of fixed degree over the simplex, De Klerk, Laurent and Parrilo \cite{KLP06} also showed the existence of a PTAS. Recently, by using Bernstein approximation and the multinomial distribution, a new proof of PTAS for fixed-degree polynomial optimization over the simplex was presented, see \cite{KLS14} for details. Indeed, in the case where feasible set is single simplex, the PTAS is particularly simple, and takes the minimum
of $f$ on the regular grid $\Delta_n(r)=\{x\in \Delta_n~|~(r+2)x\in \mathbb{N}^n\}$ for increasing values of $r\in \mathbb{N}$. Denote the minimum over the grid by
$$
f_{\Delta}^{(r)}={\rm min}\;\{f(x)~|~x\in \Delta_n(r)\}.
$$ It is obvious that the computation of $f_{\Delta}^{(r)}$ requires $|\Delta_n(r)|=\binom{n+r+1}{r+2}$ evaluations of $f$.

Moreover, we see that the regular grid mentioned above play an important role in the implement of PTAS. Several properties of the regular grid $\Delta_n(r)$ have been studied in the literature. In Bos \cite{Bos83},
the Lebesgue constant of $\Delta_n(r)$ is studied in the context of Lagrange interpolation and finite
element methods. Given a point $x\in \Delta_n$, Bomze, Gollowitzer and Yildirim \cite{BGY14} study a scheme to
find the closest point to $x$ on $\Delta_n(r)$ with respect to certain norms (including $\ell_q$-norms for finite
$q$). Furthermore, for any quadratic polynomial $f$ and $r\in \mathbb{N}$, Sagol and Yildirim \cite{SY13} and
Yildirim \cite{Y12} consider the upper bound on $f_{\rm min}$ defined by
${\rm min}_{x\in \cup_{k=0}^r\Delta_n(k)}f(x)$, and analyze the error bound. The following error bounds are known for the approximation $f_{\Delta}^{(r)}$ of $f$.
\begin{theorem}\label{Th0}

(i) \cite{BK02} For any quadratic polynomial $f$ and $r\in \mathbb{N}$, one has
$$
f_{\Delta}^{(r)}-f_{\rm min}\leq \frac{1}{r+2}(f_{\rm max}-f_{\rm min}),
$$
where $f_{\rm max}$ is the maximum value of the objective in \eqref{StQPmod}.

(ii) \cite{KLP06} For any homogeneous polynomial $f$ of degree $d\geq 2$ in \eqref{StQPmod} and $r\in \mathbb{N}\backslash\{0\}$, one has
$$
f_{\Delta}^{(d+r-2)}-f_{\rm min}\leq (1-w_r(d))\binom{2d-1}{d}d^d(f_{\rm max}-f_{\rm min}),
$$
where $w_r(d)=\frac{(d+r)!}{r!(d+r)^d}$.
\end{theorem}

The above results imply the existence of a PTAS for the corresponding optimization problems. This naturally raises the question of whether the same holds for StBQP. As far as we know, there are very few PTASs for solving standard bi-quadratic optimization problems. Indeed, the appearance of Cartesian product of two simplices in (\ref{polyOpt}) results in that the designing PTAS becomes a more complex task, which also differs from the problems considered in \cite{BK02,KLP06}.

In this paper, we focus on approximately solving StBQP, and present a quality of approximation which shows the existence of a PTAS for solving StBQP. Moreover, we prove that the proposed approach can be extended to solving standard multi-quadratic optimization problem. Some numerical examples are provided to illustrate our approach.

Some words about the notation. $\Re^n$ denotes the real Euclidean space of column vectors of length $n$, and $\mathbb{N}^n$ denotes the set of all nonnegative integer vectors of length $n$. For $\alpha=(\alpha_1,\ldots,\alpha_n)^\top\in \mathbb{N}^n$ and $d\in \mathbb{N}$, we define $|\alpha|=\alpha_1+\alpha_2+\cdots +\alpha_n$,  $\alpha!=\alpha_1!\alpha_2!\cdots \alpha_n!$ and $I(n,d)=\{\alpha\in \mathbb{N}^n~|~|\alpha|=d\}$. For two vectors $\alpha,\beta\in \Re^n$, the inequality $\alpha\leq \beta$ is coordinate-wise and means that $\alpha_i\leq \beta_i$ for every $i$.  Denote by $\mathcal{T}_{n,m}^{d,l}$ the set of all $(d,l)$-th order $\underset{d}{\underbrace{n\times\cdots\times n}} \times \underset{l}{\underbrace{m\times\cdots\times m}}$-dimensional real rectangular tensors, and $\mathcal{S}_{n,m}^{d,l}$ the set of all partially symmetric tensors in $\mathcal{T}_{n,m}^{d,l}$. Here, the meaning of partially symmetric is similar to that in (\ref{partsymm}). Specially, if $d=l=2$, $\mathcal{T}_{n,m}^{d,l}$ and $\mathcal{S}_{n,m}^{d,l}$ are simply written as $\mathcal{T}_{n,m}$ and $\mathcal{S}_{n,m}$, respectively. For $x\in \Re^n$, denoted by $x^d$ the tensor $(x_{i_1}\cdots x_{i_d})_{1\leq i_1,\ldots,i_d\leq n}$, which is a $d$-th order $n$-dimensional square tensor. For given $\mathcal{X}\in \mathcal{T}_n^d$ (the set of all $d$-th order $n$-dimensional square tensors) and $\mathcal{Y}\in \mathcal{T}_m^l$, denoted by $\mathcal{X}\otimes \mathcal{Y}$ the $(d,l)$-th order $\underset{d}{\underbrace{n\times\cdots\times n}} \times \underset{l}{\underbrace{m\times\cdots\times m}}$-dimensional real rectangular tensor with entries $\mathcal{X}_{i_1\cdots i_d}\mathcal{Y}_{j_1\cdots j_l}$ for $1\leq i_1,\ldots,i_d\leq n$ and $1\leq j_1,\ldots,j_l\leq m$. Denote by $\mathcal{E}$ the tensor of all ones in an appropriate tensor space. We denote by $\mathcal{A}\bullet\mathcal{B}$ the Frobenius inner product for tensors $\mathcal{A,B}\in \mathcal{T}_{n,m}^{d,l}$, i.e.,
$$
\mathcal{A}\bullet\mathcal{B}=\sum_{1\leq i_1,\ldots,i_d\leq n,}\sum_{1\leq j_1,\ldots,j_l\leq m}a_{i_1\ldots i_dj_1\ldots j_l}b_{i_1\ldots i_dj_1\ldots j_l}.
$$
 Let $\mathcal{H}_{n,m}^{d,l}$ denote the set of all bi-homogeneous polynomial of degrees $d\geq 2$ and $l\geq 2$, with respect to the variables $x\in \Re^n$ and $y\in \Re^m$ respectively.

\section{Preliminaries}\label{Prelim}
Recall that a set $S\subseteq \Re^n$ is said to be convex if whenever $x,y\in S$ and $t\in [0,1]$ we have $tx+(1-t)y\in S$. A set $K\subseteq \Re^n$ is said to be convex cone, if $K$ is convex and whenever $x\in K$ and $t\geq 0$ we have $tx\in K$. Let $V$ be a finite dimensional vector space equipped with a inner product $\langle \cdot,\cdot\rangle$, and let $K$ be a convex cone in $V$. Denote
$$K^*=\{w\in V~|~\langle w,k\rangle\geq 0,\forall ~k\in K \},$$
which is said to be the positive dual cone of $K$.

For a given cone $K$ and its dual cone $K^*$, we define the primal and dual pair of
conic linear programs:
$$
\begin{array}{rl}
~~~~~~~({\rm P}) ~~~~~p^*:=\inf& C\bullet X \\
{\rm s.t.}&X\in K,~ A_i\bullet X=b_i, i=1,2,\ldots,m
\end{array}
$$
and
$$\begin{array}{rl}
({\rm D}) ~~~~~d^*:=\sup&b^\top y\\
{\rm s.t.}&y\in \Re^m, C-\sum_{i=1}^m y_i A_i\in K^*.
\end{array}
$$
The following well-known conic duality theorem, see, e.g., \cite{R01}, gives the duality relations
between (P) and (D).
\begin{theorem}\label{ConicTh} (Conic duality theorem). If there exists an interior feasible solution
$X^0\in {\rm int} (K)$ of (P), and a feasible solution of (D), then $p^*=d^*$ and the supremum
in (D) is attained. Similarly, if there exist $y^0\in \Re^m$ with $C-\sum_{i=1}^m y^0_i A_i\in {\rm int}(K^*)$ and a feasible solution of (P), then $p^*=d^*$ and the infimum in (P) is attained.
\end{theorem}

We now introduce the concept of copositive rectangular tensers, which is a generalization of the concept of copositive square tensors presented in \cite{Q13} and studied in \cite{SQ14}.

\begin{definition}\label{copositivedef}
Let $\mathcal{G}=(g_{i_1\ldots i_dj_1\ldots j_l})\in \mathcal{S}_{n,m}^{d,l}$. We say that $\mathcal{G}$ is a (resp. strictly) copositive tensor, if $$\sum_{i,j=1}^n\sum_{k,l=1}^mg_{i_1\ldots i_dj_1\ldots j_l}x_{i_1}\cdots x_{i_d}y_{j_1}\cdots y_{j_l}\geq 0~({\rm resp. ~}>0)$$ for any $(x, y)\in \Re_+^n\times \Re_+^m~({\rm resp.~} (x,y)\in(\Re_+^n\backslash\{0\})\times (\Re_+^m\backslash\{0\}))$.
\end{definition}

Denote by $\mathscr{C}_{n,m}^{d,l}$ the set of all copositive tensors in $\mathcal{S}_{n,m}^{d,l}$. It is easy to see that $\mathscr{C}_{n,m}^{d,l}$ is a closed convex cone in the vector space $\mathcal{T}_{n,m}^{d,l}$. The copositive cone in $\mathcal{S}_{n,m}$ is simply defined by $\mathscr{C}_{n,m}$.

\begin{proposition}\label{Copprop}
Let $\mathcal{A}\in \mathcal{S}_{n,m}$. Then

(i) $\mathcal{A}$ is copositive, if and only if $p_{\mathcal{A}}^{\rm min}\geq 0$.

(ii) $\mathcal{A}$ is strictly copositive, if and only if $p_{\mathcal{A}}^{\rm min}> 0$.
\end{proposition}

\begin{proof}
Since $\bar x:=x/\sum_{i=1}^nx_i\in \Delta_n$ and $\bar y:=y/\sum_{j=1}^my_j\in \Delta_m$ for any $x\in \Re_+^n\backslash\{0\}$ and $y\in \Re_+^m\backslash\{0\}$, the desired results follow from Definition \ref{copositivedef}.
\qed\end{proof}

\begin{proposition}\label{Copdulity} For any given positive integers $n,m,d,l\in \mathbb{N}$ with $d,l\geq 2$, one has
$$
\mathscr{C}_{n,m}^{d,l}=(\mathscr{B}_{n,m}^{d,l})^{*},
$$
where $\mathscr{B}_{n,m}^{d,l}={\rm conv}\{x^d\otimes y^l~|~x\in \Re_+^n,~y\in \Re_+^m\}$ which is called the cone of all partially symmetric completely positive tensors in $\mathcal{T}_{n,m}^{d,l}$.
\end{proposition}

\begin{proof}
Let $\mathcal{A}\in \mathscr{C}_{n,m}^{d,l}$. For any $x^d\otimes y^l\in\mathscr{B}_{n,m}$, it is obvious that
$$\mathcal{A}\bullet(x^d\otimes y^l)=\sum_{i_1,\ldots,i_d=1}^n\sum_{j_1,\ldots,j_l=1}^ma_{i_1\ldots i_dj_1\ldots j_l}x_{i_1}\cdots x_{i_d}y_{j_1}\cdots y_{j_l}\geq 0,$$
by Definition \ref{copositivedef}. Based upon this, we may know that $\mathcal{A}\bullet\mathcal{B}\geq 0$ for any $\mathcal{B}\in \mathscr{B}_{n,m}^{d,l}$, which implies that $\mathcal{A}\in (\mathscr{B}_{n,m}^{d,l})^*$. Hence,  $\mathscr{C}_{n,m}^{d,l}\subseteq (\mathscr{B}_{n,m}^{d,l})^*$.

Now we prove $(\mathscr{B}_{n,m}^{d,l})^*\subseteq \mathscr{C}^{d,l}_{n,m}$. Let $\mathcal{A}\in(\mathscr{B}_{n,m}^{d,l})^*$. Suppose that $\mathcal{A}\not\in\mathscr{C}_{n,m}^{d,l}$. Then there exist $\bar x\in \Re_+^n$ and $\bar y\in \Re_+^m$, such that
$$\sum_{i_1,\ldots,i_d=1}^n\sum_{j_1,\ldots,j_l=1}^ma_{i_1\ldots i_dj_1\ldots j_l}\bar x_{i_1}\cdots\bar x_{i_d}\bar y_{j_1}\cdots\bar y_{j_l}< 0.$$ Take $\bar {\mathcal{Z}}=\bar {x}^d\otimes\bar {y}^l$. Then $\bar {\mathcal{Z}}\in \mathscr{B}^{d,l}_{n,m}$. From the above expression, we see that $ \mathcal{A}\bullet\bar {\mathcal{Z}}<0$, which is a contradiction. Therefore, $\mathcal{A}\in\mathscr{C}^{d,l}_{n,m}$, which implies that $(\mathscr{B}_{n,m}^{d,l})^*\subseteq \mathscr{C}^{d,l}_{n,m}$.
\qed\end{proof}

We consider the following conic optimization problem
\begin{equation}\label{conicopt}
\begin{array}{rl}
v_{\rm min}^{\rm p}={\rm min}& \mathcal{A}\bullet \mathcal{Z}\\
{\rm s.t.}&\mathcal{Z}\in \mathscr{B}_{n,m}\\
&\mathcal{E}\bullet \mathcal{Z}=1
\end{array}
\end{equation}
with $\mathscr{B}_{n,m}=\mathscr{B}_{n,m}^{2,2}$, whose dual problem is
\begin{equation}\label{duconicopt}
\begin{array}{rl}
v_{\rm max}^{\rm d}=\underset{\lambda\in \Re}{\rm max}& \lambda\\
{\rm s.t.}&\mathcal{A}-\lambda \mathcal{E}\in \mathscr{C}_{n,m}.
\end{array}
\end{equation}

It is clear that for every feasible pair $(x,y)$ of (\ref{polyOpt}), one has that $xx^\top \otimes yy^\top\in \mathscr{B}_{n,m}$ and $\mathcal{E}\bullet (xx^\top \otimes yy^\top)=1$. Hence, the problem (\ref{conicopt}) is a tensor program relaxation of (\ref{polyOpt}), which implies that $v_{\rm min}^{\rm p}\leq p_{\mathcal{A}}^{\rm min}$. However, the following theorem shows that this relaxation is exactly tight, and the solving (\ref{polyOpt}) can be converted equivalently to the solving (\ref{conicopt}).

\begin{theorem}\label{equality}
The bi-quadratic optimization \eqref{polyOpt} and conic optimization \eqref{conicopt} are equivalent, that is, \eqref{polyOpt} and \eqref{conicopt} have the same optimal value and one
optimal solution pair of \eqref{polyOpt} can be obtained from the
optimal solution of \eqref{conicopt}.
\end{theorem}

\begin{proof}
Let $\mathcal{Z}^{\ast}$ be an optimal solution of (\ref{conicopt}) with the objective value $v_{\rm min}^{\rm p}$. By the definition of $\mathscr{B}_{n,m}$, there exists a positive integer $t$ such that
$\mathcal{Z}^{\ast}=\sum_{k=1}^{t}(x^{(k)}(x^{(k)})^\top)\otimes (y^{(k)}(y^{(k)})^\top)$ with $x^{(k)}\in
\Re_+^n\backslash\{0\}$ and $y^{(k)}\in
\Re_+^m\backslash\{0\}$, $k=1,\ldots,t$. Let $\lambda_k=\left(\sum_{i=1}^{n}x_i^{(k)}\right)^2\left(\sum_{j=1}^{m}y_j^{(k)}\right)^2$. Then $\lambda_k>0$ for $k=1,2,\ldots, t$, as well as $\sum_{k=1}^{t}\lambda_k=1$ since $\mathcal{E}\bullet
\mathcal{Z}^{\ast}=1$. Moreover, it is easy to see that $\mathcal{Z}^{\ast}=\sum_{k=1}^{t}\lambda_k(\bar x^{(k)}(\bar x^{(k)})^\top)\otimes(\bar y^{(k)}(\bar y^{(k)})^\top)$, where $\bar x^{(k)}=x^{(k)}/\sum_{i=1}^nx^{(k)}_k\in \Delta_n$ and $\bar y^{(k)}=y^{(k)}/\sum_{j=1}^my^{(k)}_j\in \Delta_m$.
Since $v_{\rm
min}^{\rm p}=\mathcal{A}\bullet  \mathcal{Z}^{\ast}$,  it
follows that
$
v_{\rm min}^{\rm p}=\sum_{k=1}^{t}\lambda_k\mathcal{A}\bullet[(\bar x^{(k)}(\bar x^{(k)})^\top)\otimes(\bar y^{(k)}(\bar y^{(k)})^\top)],
$
which implies that there must exist an index, say $1$, such that
\begin{equation}\label{xbtw}
\mathcal{A}\bullet[(\bar x^{(1)}(\bar x^{(1)})^\top)\otimes(\bar y^{(1)}(\bar y^{(1)})^\top)]\leq v_{\rm
min}^{\rm p},
\end{equation}
 since $\lambda_k>0$ for $k=1,\ldots,t$ and $\sum_{k=1}^{t}\lambda_k=1$.  On the other hand, since $\bar x^{(1)}\in \Delta_n$ and $\bar y^{(1)}\in \Delta_m$, it is clear that
$$p^{\rm min}_{\mathcal{A}}\leq p_{\mathcal{A}}(\bar x^{(1)},\bar y^{(1)})=\mathcal{A}\bullet[(\bar x^{(1)}(\bar x^{(1)})^\top)\otimes(\bar y^{(1)}(\bar y^{(1)})^\top)],$$
which implies, together with (\ref{xbtw}) and the fact that $v_{\rm min}^{\rm p}\leq
p^{\rm min}_{\mathcal{A}}$, that $p^{\rm min}_{\mathcal{A}}=\mathcal{A}\bullet[(\bar x^{(1)}(\bar x^{(1)})^\top)\otimes(\bar y^{(1)}(\bar y^{(1)})^\top)]= v_{\rm min}^{\rm p}$. We obtain the desired result and complete the
proof.
\qed\end{proof}

For the linear tensor conic optimization problems (\ref{conicopt}) and (\ref{duconicopt}), by utilizing Theorem \ref{ConicTh}, we obtain the following duality result, which means, together with Theorem \ref{equality}, that there exist no polynomial time algorithms for solving (\ref{duconicopt}).

\begin{theorem}\label{DualityforCOP}
For the conic optimization problems \eqref{conicopt} and \eqref{duconicopt}, one has $v_{\rm min}^{\rm p}=v_{\rm max}^{\rm d}$.
\end{theorem}

\begin{proof}
In order to invoke Theorem \ref{ConicTh}, we have to show that there is a $\lambda\in \Re$ with $\mathcal{A}-\lambda \mathcal{E}\in  {\rm int}(\mathscr{B}_{n,m}^*)={\rm int}(\mathscr{C}_{n,m})$, and that there is a feasible solution of (\ref{conicopt}).

Take $\bar\lambda\in\Re$ such that $\bar {\mathcal{A}}:=\mathcal{A}-\bar \lambda \mathcal{E}$ is positive tensor, i.e., all entries of $\mathcal{A}-\bar \lambda \mathcal{E}$ are positive, which implies that $\bar {\mathcal{A}}$ is strictly copositive. Hence $\mathcal{A}-\bar \lambda \mathcal{E}\in {\rm int}(\mathscr{C}_{n,m})={\rm int}(\mathscr{B}_{n,m}^*)$.

On the other hand, by taking $\bar {\mathcal{Z}}=\frac{1}{n^2m^2}e^{(n)}(e^{(n)})^\top \otimes e^{(m)}(e^{(m)})^\top\in \mathscr{B}_{n,m}$, where $e^{(n)}$ and $e^{(m)}$ are the two vectors of all ones in $\Re^n$ and $\Re^m$ respectively, we may verify that $\mathcal{E}\bullet\bar {\mathcal{Z}}=1$, which means $\bar { \mathcal{Z}}$ is a feasible solution of (\ref{conicopt}). By Theorem \ref{ConicTh}, we obtain the desired result.
\qed\end{proof}

\section{Approximation of copositive tensor cones}\label{Sec3}

Copositive programming is a useful tool in dealing with all sorts of optimization problems. However, it is well-known that the problem of checking whether a symmetric matrix belongs to the cone of copositive matrices or not is co-NP-complete
\cite{MK87}. The appearance of copositive tensor in (\ref{duconicopt}) results in the considered problem becomes a more complex task, since copositive tensor hides more complex structures than matrix in terms of computational solvability. In this section, we focus attention on studying how to approximate the copositive tensor cone $\mathscr{C}_{n,m}$.

By virtue of $p_{\mathcal{G}}(x,y)$ in (\ref{polyOpt}) with $\mathcal{G}\in \mathcal{S}_{n,m}$, we define
\begin{equation}\label{Grszw}
 P_{\mathcal{G}}^{(s,r)}(x,y):=p_{\mathcal{G}}(x,y)\left(\sum_{i=1}^nx_i\right)^s\left(\sum_{j=1}^my_j\right)^r,
\end{equation}
where $s$ and $r$ are any given non-negative integers. We consider when $P_{\mathcal{G}}^{(s,r)}(x,y)$ has no negative coefficients. It is clear that the set of all partially symmetric tensors $\mathcal{G}$ satisfying this condition forms a convex cone in $\mathcal{S}_{n,m}$.

\begin{definition}\label{DKdef}
The convex cone $\mathscr{C}_{n,m}^{s,r}$ consists of the tensors in $\mathcal{S}_{n,m}$ for which $P_{\mathcal{G}}^{(s,r)}(x,y)$ in (\ref{Grszw}) has no negative coefficients.
\end{definition}
Obviously, these cones are contained in each other: $\mathscr{C}_{n,m}^{s,r}\subseteq \mathscr{C}_{n,m}^{s+1,r}\subseteq \mathscr{C}_{n,m}^{s+1,r+1}$ for all non-negative integers $s$ and $r$. The approximation of $\mathscr{C}_{n,m}^{s,r}$ to $\mathscr{C}_{n,m}$ is essentially by examination of the proof of the following theorem, which is a bi-quadratic version of the famous theorem of P\'{o}lya \cite{HLP52,Poly74} (see also Powers and Reznick \cite{PR01}).

\begin{theorem}\label{GenPoly}
Let $\mathcal{G}\in \mathcal{S}_{n,m}$ and $p_{\mathcal{G}}(x,y)$ be a bi-quadratic form defined in \eqref{polyOpt}. Suppose that $p_{\mathcal{G}}(x,y)$ is positive on the Cartesian product of two simplices $\Delta_n\times \Delta_m$, i.e., $p_{\mathcal{G}}^{\rm min}>0$. Then the polynomial $p_{\mathcal{G}}(x,y)(\sum_{i=1}^nx_i)^r(\sum_{j=1}^my_j)^s$ has non-negative coefficients for all sufficiently large integers $r$ and $s$.
\end{theorem}

\begin{proof}
In terms of $I(n,2)$ and $I(m,2)$, we rewrite $p_{\mathcal{G}}(x,y)$ in (\ref{polyOpt}) as
\begin{equation}\label{bpxy}
p_{\mathcal{G}}(x,y)=\sum_{\alpha\in I(n,2),\beta\in I(m,2)}\bar p_{\alpha\beta}\frac{x^\alpha y^\beta}{\alpha!\beta!},
\end{equation}
that is, $\bar p_{\alpha\beta}=p_{\alpha\beta}\alpha!\beta!$. Define
$$
\phi(x,y,t,s)=\sum_{\alpha\in I(n,2),\beta\in I(m,2)}\bar p_{\alpha\beta}\frac{(x_1)_t^{\alpha_1}\cdots (x_n)_t^{\alpha_n} (y_1)_s^{\beta_1}\cdots (y_m)_s^{\beta_m}}{\alpha!\beta!},
$$
where $(a)_t^q:=a(a-t)\cdots (a-(q-1)t)$ with $(a,t,q)\in \Re\times \Re_+\times \mathbb{N}$.
It is clear that $\phi(x,y,0,0)=p(x,y)$, and $\phi$ is continuous in $\Delta_n\times \Delta_m\times [0,1]\times [0,1]$. Consequently, by the given condition, there exist positive numbers $\mu$ and $\bar \varepsilon\in [0,1]$, such that
\begin{equation}\label{phPos}
\phi(x,y,t,s)>\phi(x,y,0,0)-(1/2)\mu\geq (1/2)\mu>0
\end{equation}
for $(x,y,t,s)\in \Delta_n\times \Delta_m\times [0,\bar\varepsilon]\times [0,\bar \varepsilon]$.

On the other hand, we also have that for any $s,r\geq 2$,
\begin{equation}\label{sumxy}
\left(\sum_{i=1}^nx_i\right)^{r-2}\left(\sum_{j=1}^my_j\right)^{s-2}=(r-2)!(s-2)!\sum_{\gamma\in I(n,r-2),\lambda\in I(m,s-2)}\frac{x^{\gamma}y^\lambda}{\gamma!\lambda!}.
\end{equation}
By multiplying (\ref{bpxy}) and (\ref{sumxy}), we obtain
$$
p_{\mathcal{G}}(x,y)\left(\sum_{i=1}^nx_i\right)^{r-2}\left(\sum_{j=1}^my_j\right)^{s-2}=(r-2)!(s-2)!\underset{\beta\in I(m,2)}{\sum_{\alpha\in I(n,2)}}\underset{~\lambda\in I(m,s-2)}{\sum_{~\gamma\in I(n,r-2)}}\bar p_{\alpha\beta}\frac{x^{\alpha+\gamma}y^{\beta+\lambda}}{\gamma!\lambda!\alpha!\beta!}.
$$
Denote $\alpha+\gamma=\xi$ and $\beta+\lambda=\zeta$. Then, it holds that $\xi\in I(n,r)$ and $\zeta\in I(m,s)$. Moreover, it is not difficult to see that
\begin{equation}\label{xypxy}
p_{\mathcal{G}}(x,y)\left(\sum_{i=1}^nx_i\right)^{r-2}\left(\sum_{j=1}^my_j\right)^{s-2}=\displaystyle(r-2)!(s-2)!\underset{\zeta\in I(m,s)}{\sum_{\xi\in I(n,r)}}\frac{x^{\xi}y^{\zeta}}{\xi!\zeta!}\underset{\alpha\leq \xi,\beta\leq \zeta}{\underset{~\beta\in I(m,2)}{\sum_{~\alpha\in I(n,2)}}}\bar p_{\alpha\beta}\binom{\xi}{\alpha}\binom{\zeta}{\beta},
\end{equation}
where $\binom{\xi}{\alpha}=\frac{\xi!}{\alpha!(\xi-\alpha)!}$ and $\binom{\zeta}{\beta}=\frac{\zeta!}{\beta!(\zeta-\beta)!}$. Since $\binom{b}{a}=0$ for any $a, b\in \mathbb{N}$ with $b< a$, by (\ref{xypxy}), we have
\begin{equation*}\label{xypxy1}
\begin{array}{l}
\displaystyle p_{\mathcal{G}}(x,y)\left(\sum_{i=1}^nx_i\right)^{r-2}\left(\sum_{j=1}^my_j\right)^{s-2}\\
~~~~=\displaystyle(r-2)!(s-2)!\underset{\zeta\in I(m,s)}{\sum_{\xi\in I(n,r)}}\frac{x^{\xi}y^{\zeta}}{\xi!\zeta!}\underset{~\beta\in I(m,2)}{\sum_{~\alpha\in I(n,2)}}\bar p_{\alpha\beta}\binom{\xi}{\alpha}\binom{\zeta}{\beta}\\
~~~~=\displaystyle(r-2)!(s-2)!r^2s^2\underset{\zeta\in I(m,s)}{\sum_{\xi\in I(n,r)}}\phi(\xi/r,\zeta/s,1/r,1/s)\frac{x^{\xi}y^{\zeta}}{\xi!\zeta!}.
\end{array}
\end{equation*}
Since, $\phi$ here is positive for sufficiently large $r$ and $s$  by (\ref{phPos}),  we obtain the desired result and complete the proof.
\qed\end{proof}

For every $\mathcal{M}\in \mathscr{C}_{n,m}$, we claim that $\mathcal{M}\in \mathscr{C}_{n,m}^{s,r}$ for sufficiently large $s$ and $r$. In fact, it is clear that $\bar {\mathcal{M}}_{t}:=\mathcal{M}+t\mathcal{E}$ is strictly copositive for any $t>0$, which implies $p_{\bar {\mathcal{M}}_{t}}^{\rm min}>0$ by Proposition \ref{Copprop}. Consequently, by Theorem \ref{GenPoly}, we know that $\bar {\mathcal{M}}_{t}$ lies in some cone $\mathscr{C}_{n,m}^{s,r}$ for $s,r$ sufficiently large. Since $\mathscr{C}_{n,m}^{s,r}$ is closed for any fixed $s$ and $r$, by letting $t\rightarrow 0$, we know that $\mathcal{M}\in \mathscr{C}_{n,m}^{s,r}$ for  sufficiently large $s$ and $r$.

For given $\alpha\in \Re^n$, define
\begin{equation}\label{DDaapp}
c(\alpha)=\left\{
\begin{array}{cl}
\displaystyle\frac{|\alpha|!}{\alpha!},&{\rm if}~\alpha\in \mathbb{N}^n,\\
0,&{\rm ~otherwise.}
\end{array}
\right.
\end{equation}
 Moreover, for given $\xi\in \Re^n$, $\zeta\in \Re^m$ and $\mathcal{G}\in \mathcal{S}_{n,m}$, define
\begin{equation}\label{Axzijkl}
\bar Q_{\xi\zeta}(\mathcal{G})=\sum_{i,j=1}^n\sum_{k,l=1}^mg_{ijkl}c(\xi^{(n)}(i,j))c(\zeta^{(m)}(k,l)),
\end{equation}
where $\xi^{(n)}(i,j)=\xi-e^{(n)}_i-e^{(n)}_j$ for $e^{(n)}_i, e^{(n)}_j\in \Re^n$, and $\zeta^{(m)}(k,l)=\zeta-e^{(m)}_k-e^{(m)}_l$ for $e^{(m)}_k, e^{(m)}_l\in \Re^m$. Here, $e^{(n)}_i$ is the $i$-th column vector of the identity matrix in $\Re^{n\times n}$.

By the multinomial law, it holds that
\begin{equation}\label{Gsrmulnomial}
\begin{array}{l}
P_{\mathcal{G}}^{(s,r)}(x,y)\\
=\displaystyle\left(\sum_{i,j=1}^n\sum_{k,l=1}^mg_{ijkl}x_{i}x_{j}y_{k}y_l\right)\left(\sum_{p=1}^nx_p\right)^s\left(\sum_{q=1}^my_q\right)^r\\
=\displaystyle\sum_{\alpha\in I(n,s)}\sum_{\beta\in I(m,r)}\sum_{i,j=1}^n\sum_{k,l=1}^mg_{ijkl}\frac{s!r!}{\alpha!\beta!}x_{i}x_{j}y_{k}y_lx^{\alpha}y^{\beta}\\
=\displaystyle\sum_{\alpha\in I(n,s)}\sum_{\beta\in I(m,r)}\sum_{i,j=1}^n\sum_{k,l=1}^mg_{ijkl}\frac{s!r!}{\alpha!\beta!}x^{\alpha+e^{(n)}_i+e^{(n)}_j}y^{\beta+e^{(m)}_k+e^{(m)}_l}\\
=\displaystyle\sum_{\xi\in I(n,s+2), \xi^{(n)}(i,j)\geq 0}~\sum_{\zeta\in I(m,r+2),\zeta^{(m)}(k,l)\geq 0}\left[\sum_{i,j=1}^n\sum_{k,l=1}^m g_{ijkl}\frac{s!r!}{\xi^{(n)}(i,j)!\zeta^{(m)}(k,l)!}\right]x^{\xi}y^{\zeta}\\
=\displaystyle\sum_{\xi\in I(n,s+2)}\sum_{\zeta\in I(m,r+2)}\left[\sum_{i,j=1}^n\sum_{k,l=1}^mg_{ijkl}c(\xi^{(n)}(i,j))c(\zeta^{(m)}(k,l))\right]x^{\xi}y^{\zeta}\\
=\displaystyle\sum_{\xi\in I(n,s+2)}\sum_{\zeta\in I(m,r+2)}\bar Q_{\xi\zeta}(\mathcal{G})x^{\xi}y^{\zeta},
\end{array}
\end{equation}
where the last second equality is due to (\ref{DDaapp}), and the last equality comes from (\ref{Axzijkl}). From (\ref{Gsrmulnomial}), we see that $\bar Q_{\xi\zeta}(\mathcal{G})$ as given by (\ref{Axzijkl}), are exactly the coefficients of $P_{\mathcal{G}}^{(s,r)}$.

For $\mathcal{G}\in \mathcal{S}_{n,m}$, denote by $A^{(i)}~(i=1,\ldots,n)$ the $m\times m$ symmetric matrix with entries being $g_{iikl}~(k,l=1,2,\ldots,m)$, $B^{(k)}~(k=1,\ldots,m)$ the $n\times n$ symmetric matrix with entries being $g_{ijkk}~(i,j=1,2,\ldots,n)$, and $C$ the $n\times m$ matrix with entries being $g_{iikk}~(i=1,\ldots,n,~k=1,2,\ldots,m)$. The following auxiliary result simplifies the expressions $\bar Q_{\xi\zeta}(\mathcal{G})$ considerably.
\begin{lemma}\label{lemm2.3}
Let $\mathcal{G}\in \mathcal{S}_{n,m}$, $\xi\in I(n,s+2)$ and $\zeta\in I(m,r+2)$. Let $\bar Q_{\xi\zeta}(\mathcal{G})$ be defined in \eqref{Axzijkl}. Then,
\begin{equation}\label{bargxzsr}
\begin{array}{r}
\displaystyle\bar Q_{\xi\zeta}(\mathcal{G})=\frac{c(\xi)c(\zeta)}{(s+2)(s+1)(r+2)(r+1)}\left((\mathcal{G}\xi\xi^\top)\bullet(\zeta\zeta^\top)-\sum_{i=1}^n\xi_i(\zeta^\top A^{(i)}\zeta)\right.\\
\displaystyle\left.-\sum_{k=1}^m\zeta_k(\xi^\top B^{(k)}\xi)+\xi^\top C\zeta\right),
\end{array}
\end{equation}
where $c(\cdot)$ is defined in \eqref{DDaapp}.
\end{lemma}
\begin{proof}
It is easy to verify that, if $\xi^{(n)}(i,j)\in \Re^n\backslash \mathbb{N}^n$, then $c(\xi^{(n)}(i,j))=0$; otherwise, we have
$$c(\xi^{(n)}(i,j))=\left\{
\begin{array}{ll}
\displaystyle c(\xi)\frac{\xi_i(\xi_i-1)}{(s+2)(s+1)},&{\rm if ~}i=j,\\
\displaystyle c(\xi)\frac{\xi_i\xi_j}{(s+2)(s+1)},&{\rm otherwise.}
\end{array}
\right.$$
Similar result on $c(\zeta^{(m)}(k,l))$ also holds. Consequently, by (\ref{Axzijkl}), it holds that
$$
\begin{array}{lll}
\bar Q_{\xi\zeta}(\mathcal{G})&=&\displaystyle\sum_{1\leq i\leq n}\sum_{1\leq k\leq m}\frac{\xi_i(\xi_i-1)\zeta_k(\zeta_k-1)}{(s+2)(s+1)(r+2)(r+1)}g_{iikk}c(\xi)c(\zeta)\\
&&+\displaystyle\sum_{1\leq i\leq n}\sum_{1\leq k\neq l\leq m}\frac{\xi_i(\xi_i-1)\zeta_k\zeta_l}{(s+2)(s+1)(r+2)(r+1)}g_{iikl}c(\xi)c(\zeta)\\
&&+\displaystyle\sum_{1\leq i\neq j\leq n}\sum_{1\leq k\leq m}\frac{\xi_i\xi_j\zeta_k(\zeta_k-1)}{(s+2)(s+1)(r+2)(r+1)}g_{ijkk}c(\xi)c(\zeta)\\
&&+\displaystyle\sum_{1\leq i\neq j\leq n}\sum_{1\leq k\neq l\leq m}\frac{\xi_i\xi_j\zeta_k\zeta_l}{(s+2)(s+1)(r+2)(r+1)}g_{ijkl}c(\xi)c(\zeta)\\
&=&\displaystyle\frac{c(\xi)c(\zeta)}{(s+2)(s+1)(r+2)(r+1)}\left\{\sum_{1\leq i,j\leq n}\sum_{1\leq k,l\leq m}g_{ijkl}\xi_i\xi_j\zeta_k\zeta_l-\sum_{1\leq i\leq n}\sum_{1\leq k,l\leq m}g_{iikl}\xi_i\zeta_k\zeta_l\right.\\
&&\displaystyle\left.-\sum_{1\leq i,j\leq n}\sum_{1\leq k\leq m}g_{ijkk}\xi_i\xi_j\zeta_k+\sum_{1\leq i\leq n}\sum_{1\leq k\leq m}g_{iikk}\xi_i\zeta_k\right\},
\end{array}
$$
which exactly corresponds to (\ref{bargxzsr}).
\qed\end{proof}

It is not difficult to see that, when $\mathcal{G}=\mathcal{E}\in \mathcal{S}_{n,m}$, for any $\xi\in I(n,s+2)$ and $\zeta\in I(m,r+2)$, one has
\begin{equation}\label{Ezxcoeff}
\bar Q_{\xi\zeta}(\mathcal{E})=c(\xi)c(\zeta),
  \end{equation}
since in this case, by a simple computation, we have
$$\sum_{1\leq i,j\leq n}\sum_{1\leq k,l\leq m}g_{ijkl}\xi_i\xi_j\zeta_k\zeta_l=(s+2)^2(r+2)^2,~~~\sum_{1\leq i\leq n}\sum_{1\leq k,l\leq m}g_{iikl}\xi_i\zeta_k\zeta_l=(s+2)(r+2)^2,$$
and
$$\sum_{1\leq i,j\leq n}\sum_{1\leq k\leq m}g_{ijkk}\xi_i\xi_j\zeta_k=(s+2)^2(r+2), ~~~\sum_{1\leq i\leq n}\sum_{1\leq k\leq m}g_{iikk}\xi_i\zeta_k=(s+2)(r+2).$$

Based upon Lemma \ref{lemm2.3}, we can immediately derive a polyhedral representation of the cones $\mathscr{C}_{n,m}^{s,r}$.

\begin{theorem}\label{Th3}
For all $n, m,s, r\in \mathbb{N}$, one has
$$
\begin{array}{lll}
\mathscr{C}_{n,m}^{s,r}&=&\displaystyle\left\{\mathcal{G}\in \mathcal{S}_{n,m}~|~(\mathcal{G}\xi\xi^\top)\bullet(\zeta\zeta^\top)-\sum_{i=1}^n\xi_i(\zeta^\top A^{(i)}\zeta)\right.\\
&&\displaystyle\left.~~~-\sum_{k=1}^m\zeta_k(\xi^\top B^{(k)}\xi)+\xi^\top C\zeta\geq 0,~{\rm for ~all~}\xi\in I(n,s+2),\zeta\in I(m,r+2)\right\},
\end{array}
$$
where $A^{(i)}~(i=1,\ldots,n)$, $B^{(k)}~(k=1,\ldots,m)$ and $C$ are defined before Lemma \ref{lemm2.3}.
\end{theorem}

\begin{proof}
It follows from (\ref{Gsrmulnomial}) and (\ref{bargxzsr}). The proof is completed.
\qed\end{proof}

From Theorem \ref{Th3}, we know that
$$
\begin{array}{lll}
\mathscr{C}_{n,m}^{s,r}&=&\displaystyle\{\mathcal{G}\in \mathcal{S}_{n,m}~|~\mathcal{G}\bullet[(\xi\xi^\top)\otimes(\zeta\zeta^\top)-{\rm Diag}(\xi)\otimes(\zeta\zeta^\top)-(\xi\xi^\top)\otimes{\rm Diag}(\zeta)\\
&&+{\rm Diag}(\xi)\otimes{\rm Diag}(\zeta)]\geq 0,~{\rm for ~all~}\xi\in I(n,s+2),\zeta\in I(m,r+2)\},
\end{array}
$$
and hence
\begin{equation*}\label{Cdaohj}
\begin{array}{lll}
(\mathscr{C}_{n,m}^{s,r})^*&\supset&\displaystyle\{(\xi\xi^\top)\otimes(\zeta\zeta^\top)-{\rm Diag}(\xi)\otimes(\zeta\zeta^\top)-(\xi\xi^\top)\otimes{\rm Diag}(\zeta)\\
&&+{\rm Diag}(\xi)\otimes{\rm Diag}(\zeta)~|~\xi\in I(n,s+2),\zeta\in I(m,r+2)\}\neq\emptyset.
\end{array}
\end{equation*}

\section{Quality of approximation}\label{Sec4}

For any given non-negative integers $s,r\in \mathbb{N}$, define
\begin{equation}\label{csrprimal}
{\rm min}\left\{ \mathcal{A}\bullet \mathcal{X} ~|~\mathcal{E}\bullet \mathcal{X}=1,\mathcal{X}\in (\mathscr{C}_{n,m}^{s,r})^*\right\}
\end{equation}
which has dual problem
\begin{equation}\label{csrdual}
{\rm max}\left\{\lambda~|~\mathcal{A}-\lambda \mathcal{E}\in \mathscr{C}_{n,m}^{s,r},\lambda\in \Re\right\}.
\end{equation}

By taking $\bar \xi=(s+2,0,\ldots, 0)^\top\in I(n,s+2)$ and $\bar \zeta=(r+2,0,\ldots,0)^\top\in I(m,r+2)$, we may find a feasible point of (\ref{csrprimal}). On the other hand, it is obvious that there exists a $\bar \lambda\in \Re$ such that all elements of $\mathcal{A}-\bar \lambda \mathcal{E}$ are positive, which implies $\mathcal{A}-\bar \lambda \mathcal{E}\in {\rm int}(\mathscr{C}_{n,m}^{s,r})$. By Theorem \ref{ConicTh}, we know that the optimal value, denoted by $p_{\mathscr{C}}^{(s,r)}$, of (\ref{csrprimal}) equals to one of (\ref{csrdual}). Moreover, we know that problem (\ref{csrdual}) is a relaxation of problem (\ref{duconicopt}), in which the copositive cone $\mathscr{C}_{n,m}$ is approximated by $\mathscr{C}_{n,m}^{s,r}$ in the sense that every $\mathcal{M}\in \mathscr{C}_{n,m}$ must lie in some cone $\mathscr{C}_{n,m}^{s,r}$ for sufficiently large $s$ and $r$. It therefore follows that $p_{\mathscr{C}}^{(s,r)}\leq p^{\rm min}_{\mathcal{A}}$ for all sufficiently large $s,r$. We now provide an alternative representation of $p_{\mathscr{C}}^{(s,r)}$. This representation uses the following two rational grids which approximate the standard simplices $\Delta_n$ and $\Delta_m$ in (\ref{polyOpt}):
\begin{equation*}\label{Deltnm}
 \Delta_n(s):=\left\{x\in \Delta_n~|~(s+2)x\in \mathbb{N}^n\right\}~~~{\rm and}~~~
 \Delta_m(r):=\left\{y\in \Delta_m~|~(r+2)y\in \mathbb{N}^m\right\}.
\end{equation*}
Consequently, a natural approximation of problem (\ref{polyOpt}) would be
\begin{equation}\label{srapprox}
p^{(s,r)}_{\Delta}:={\rm min}\left\{p_{\mathcal{A}}(x,y)~|~x\in \Delta_n(s),y\in \Delta_m(r)\right\},
\end{equation}
which satisfies $p^{(s,r)}_{\Delta}\geq p^{\rm min}_{\mathcal{A}}$.

We now establish the connection between $p_{\mathscr{C}}^{(s,r)}$ and $p^{(s,r)}_{\Delta}$.

\begin{theorem}\label{Th4}
For any given $s, r\in \mathbb{N}$. We consider the rational discretization $\Delta_n(s)\times\Delta_m(r)$ of $\Delta_n\times \Delta_m$. If $\mathcal{A}\in \mathcal{S}_{n,m}$, then
\begin{equation}\label{AbiQ}
 p_{\mathscr{C}}^{(s,r)}=\frac{(s+2)(r+2)}{(s+1)(r+1)}~{\rm min}\left\{\bar p_{\mathcal{A}}(x,y)~|~x\in \Delta_n(s),y\in \Delta_m(r)\right\},
\end{equation}
where $$
\bar p_{\mathcal{A}}(x,y):=p_{\mathcal{A}}(x,y)-\sum_{i=1}^nx_iy^\top \tilde A^{(i)}y-\sum_{k=1}^my_kx^\top \hat A^{(k)}x+x^\top \check Ay
$$
with
\begin{numcases}{}
\tilde A^{(i)}=\frac{1}{s+2}(a_{iikl})_{1\leq k,l\leq m},\;\;(i=1,\ldots,n), \nn\\
\hat A^{(k)}=\frac{1}{r+2}(a_{ijkk})_{1\leq i,j\leq n},\;\;(k=1,\ldots,m), \nn \\
\check A=\frac{1}{(s+2)(r+2)}(a_{iikk})_{1\leq i\leq n,l\leq k\leq m}.\nn
\end{numcases}
\end{theorem}

\begin{proof}
By putting $\mathcal{G}=\mathcal{A}-\lambda \mathcal{E}$, we have
$$
\begin{array}{lll}
\bar Q_{\xi\zeta}(\mathcal{G})&=&\bar Q_{\xi\zeta}(\mathcal{A})-\lambda\bar Q_{\xi\zeta}(\mathcal{E})\\
&=&\displaystyle c(\xi)c(\zeta)\left\{\frac{1}{(s+2)(s+1)(r+2)(r+1)}\left((\mathcal{A}\xi\xi^\top)\bullet(\zeta\zeta^\top)-(s+2)\sum_{i=1}^n\xi_i(\zeta^\top \tilde{A}^{(i)}\zeta)\right.\right.\\
&&\displaystyle\left.\left.-(r+2)\sum_{k=1}^m\zeta_k(\xi^\top \hat{A}^{(k)}\xi)+(s+2)(r+2)\xi^\top \check{A}\zeta\right)-\lambda\right\},
\end{array}
$$
where the second equality comes from (\ref{bargxzsr}) and (\ref{Ezxcoeff}). Consequently, by (\ref{csrdual}) and Theorem \ref{Th3}, we have
$$
\begin{array}{lll}
p_{\mathscr{C}}^{(s,r)}&=&\displaystyle \frac{1}{(s+2)(s+1)(r+2)(r+1)} {\rm min}\left\{\left((\mathcal{A}\xi\xi^\top)\bullet(\zeta\zeta^\top)-(s+2)\sum_{i=1}^n\xi_i(\zeta^\top \tilde{A}^{(i)}\zeta)\right.\right.\\
&&~~~\displaystyle\left.\left.-(r+2)\sum_{k=1}^m\zeta_k(\xi^\top\hat{A}^{(k)}\xi)+(s+2)(r+2)\xi^\top \check{A}\zeta\right)~|~\xi\in I(n,s+2),~\zeta\in I(m,r+2)\right\},
\end{array}
$$
which implies that (\ref{AbiQ}) holds, by putting $x=\frac{1}{s+2}\xi$ and $y=\frac{1}{r+2}\zeta$. We complete the proof.
\qed\end{proof}

From Theorem \ref{Th4}, we see that $p^{(s,r)}_{\mathscr{C}}$ can be obtained by only function evaluations at points on the rational grid $\Delta_n(s)\times\Delta_m(r)$, which could help us obtain $p^{(s,r)}_{\mathscr{C}}$ more easily than doing so by (\ref{csrdual}) directly.

Now we present our main result in this paper.

\begin{theorem}\label{Th5} Let $p_{\mathcal{A}}^{\rm max}={\rm max}\{p_{\mathcal{A}}(x,y)~|~x\in \Delta_n,~y\in \Delta_m\}$. One has
\begin{equation*}\label{approxrat1}
p_{\mathcal{A}}^{\rm min}-p^{(s,r)}_{\mathscr{C}}\leq \frac{s+r+4}{(s+1)(r+1)}\left(p_{\mathcal{A}}^{\rm max}-p_{\mathcal{A}}^{\rm min}\right)
\end{equation*}
and
\begin{equation*}\label{approxrat2}
p^{(s,r)}_{\Delta}-p_{\mathcal{A}}^{\rm min}\leq \frac{s+r+4}{(s+2)(r+2)}\left(p_{\mathcal{A}}^{\rm max}-p_{\mathcal{A}}^{\rm min}\right).
\end{equation*}
\end{theorem}

\begin{proof}
It is obvious that $a_{iikk}=p_{\mathcal{A}}(e_i^{(n)},e_k^{(m)})\geq p_{\mathcal{A}}^{\rm min}$ for every $i=1,\ldots,n$ and $k=1,\ldots,m$, which implies that \begin{equation}\label{Cxyik}
(s+2)(r+2)x^\top \check{A}y\geq p_{\mathcal{A}}^{\rm min}, ~~~\forall~(x,y)\in \Delta_n(s)\times \Delta_m(r).
\end{equation}
Moreover, it is easy to see that for $i=1,\ldots,n$, $k=1,\ldots,m$ and $(x,y)\in \Re^n\times \Re^m$, we have $(s+2)y^\top \tilde{A}^{(i)}y=p_{\mathcal{A}}(e_i^{(n)},y)$ and $(r+2)x^\top \hat{A}^{(k)}x=p_{\mathcal{A}}(x,e_k^{(m)})$, which implies that
\begin{equation}\label{Aixy}
\begin{array}{lll}
(s+2){\rm max}\{y^\top \tilde{A}^{(i)}y~|~y\in \Delta_m(r)\}&=&{\rm max}\{p_{\mathcal{A}}(e_i^{(n)},y)~|~y\in \Delta_m(r)\}\\
&\leq& {\rm max}\{p_{\mathcal{A}}(x,y)~|~x\in \Delta_n(s),~y\in \Delta_m(r)\}\\
&=&p_{\mathcal{A}}^{\rm max}
\end{array}
\end{equation}
for $i=1,\ldots,n$, and
\begin{equation}\label{Bkxy}
\begin{array}{lll}
(r+2){\rm max}\{x^\top \hat{A}^{(k)}x~|~x\in \Delta_n(s)\}&=&{\rm max}\{p_{\mathcal{A}}(x,e_k^{(m)})~|~x\in \Delta_n(s)\}\\
&\leq& {\rm max}\{p_{\mathcal{A}}(x,y)~|~x\in \Delta_n(s),~y\in \Delta_m(r)\}\\
&=&p_{\mathcal{A}}^{\rm max}
\end{array}
\end{equation}
for $k=1,\ldots,m$. Consequently, by (\ref{Aixy}) and (\ref{Bkxy}), we know that
\begin{equation}\label{xyABinq}
\sum_{i=1}^nx_iy^\top A^{(i)}y\leq  p_{\mathcal{A}}^{\rm max}/(s+2)~~{\rm and}~~ \sum_{k=1}^ny_kx^\top \hat{A}^{(k)}x\leq  p_{\mathcal{A}}^{\rm max}/(r+2),~~~~\forall~x\in \Delta_n(s),~y\in \Delta_m(r).
\end{equation}
By Theorem \ref{Th4}, (\ref{Cxyik}) and (\ref{xyABinq}), it holds that
$$
\begin{array}{lll}
p_{\mathscr{C}}^{(s,r)}&=&\displaystyle\frac{(s+2)(r+2)}{(s+1)(r+1)}~{\rm min}\left\{\bar p_{\mathcal{A}}(x,y)~|~x\in \Delta_n(s),y\in \Delta_m(r)\right\}\\
&\geq&\displaystyle\frac{(s+2)(r+2)}{(s+1)(r+1)}\left(p_{\mathcal{A}}^{\rm min}-\frac{1}{s+2}p_{\mathcal{A}}^{\rm max}-\frac{1}{r+2}p_{\mathcal{A}}^{\rm max}+\frac{1}{(s+2)(r+2)}p_{\mathcal{A}}^{\rm min}\right)\\
&=&\displaystyle\frac{(s+2)(r+2)+1}{(s+1)(r+1)}p_{\mathcal{A}}^{\rm min}-\frac{s+r+4}{(s+1)(r+1)}p_{\mathcal{A}}^{\rm max}.
\end{array}
$$
The first result follows.

The second inequality is derived by a similar way. By Theorem \ref{Th4}, there exists $(\bar x,\bar y)\in \Delta_n(s)\times \Delta_m(r)$ such that
$$
\begin{array}{lll}
p_{\mathscr{C}}^{(s,r)}&=&\displaystyle\frac{(s+2)(r+2)}{(s+1)(r+1)}\bar p_{\mathcal{A}}(\bar x,\bar y)\\
&\geq &\displaystyle\frac{(s+2)(r+2)}{(s+1)(r+1)}\left(p_{\Delta}^{(s,r)}-\sum_{i=1}^n\bar x_i\bar y^\top \tilde{A}^{(i)}\bar y-\sum_{k=1}^m\bar y_k\bar x^\top \hat{A}^{(k)}\bar x+\bar x^\top \check{A}\bar y\right),
\end{array}
$$
which implies
\begin{equation}\label{pminpmax}
\begin{array}{lll}
p_{\Delta}^{(s,r)}&\leq& \displaystyle\frac{(s+1)(r+1)}{(s+2)(r+2)}p_{\mathscr{C}}^{(s,r)}+\sum_{i=1}^n\bar x_i\bar y^\top \tilde{A}^{(i)}\bar y+\sum_{k=1}^m\bar y_k\bar x^\top \hat{A}^{(k)}\bar x-\bar x^\top \check{A}\bar y\\
&\leq&\displaystyle\frac{(s+1)(r+1)}{(s+2)(r+2)}p_{\mathcal{A}}^{\rm min}-\frac{1}{(s+2)(r+2)}p_{\mathcal{A}}^{\rm min}+\frac{1}{s+2}p_{\mathcal{A}}^{\rm max}+\frac{1}{r+2}p_{\mathcal{A}}^{\rm max}.
\end{array}
\end{equation}
By (\ref{pminpmax}), we obtain
$$
p_{\Delta}^{(s,r)}-p_{\mathcal{A}}^{\rm min}\leq \displaystyle\frac{s+r+4}{(s+2)(r+2)}\left(p_{\mathcal{A}}^{\rm max}-p_{\mathcal{A}}^{\rm min}\right).
$$
We obtain the desired result and complete the proof.
\qed\end{proof}

In the rest of this section, we consider the following standard multi-quadratic optimization problem (StMQP)
\begin{equation}\label{BQPS}
\begin{array}{rl}
p^{\rm min}_{\mathcal{A}}={\rm min}& p_{\mathcal{A}}\left(x^{(1)},\ldots,x^{(d)}\right):=
\underset{i_1,j_1=1}{\overset{n_1}\sum}\cdots\underset{i_d,j_d=1}{\overset{n_d}\sum} a_{i_1j_1\ldots i_dj_d}x^{(1)}_{i_1}x^{(1)}_{j_1}\cdots x^{(d)}_{i_d}x^{(d)}_{j_d}\\
{\rm s.t.}&\left(x^{(1)},\ldots,x^{(d)}\right)\in \Delta_{n_1}\times\cdots\times\Delta_{n_d},
\end{array}
\end{equation}
where $\mathcal{A}=(a_{i_1j_1\ldots i_dj_d})_{1\leq i_k,j_k\leq n_k,1\leq k\leq d}$.

For any given $r_1,\ldots, r_d\in \mathbb{N}$, define
$$
p_{\mathscr{C}}^{(r_1,\ldots,r_d)}={\rm min}\left\{\mathcal{A}\bullet \mathcal{X}~|~\mathcal{X}\bullet \mathcal{E}=1,\mathcal{X}\in (\mathscr{C}_{n_1,\ldots,n_d}^{r_1,\ldots,r_d})^*\right\},
$$
whose dual problem is
$$
p_{\mathscr{C}}^{(r_1,\ldots,r_d)}={\rm max}\left\{\lambda\in \Re~|~\mathcal{A}-\lambda \mathcal{E}\in \mathscr{C}_{n_1,\ldots,n_d}^{r_1,\ldots,r_d}\right\}.
$$
Similarly as for StBQP, define
$$
p_{\Delta}^{(r_1,\ldots,r_d)}={\rm min}\left\{p_{\mathcal{A}}(x^{(1)},\ldots,x^{(d)})~|~x^{(k)}\in \Delta_{n_k}(r_k), 1\leq k\leq d\right\},
$$
where $\mathscr{C}_{n_1,\ldots,n_d}^{r_1,\ldots,r_d}$ is defined  similarly as Definition \ref{DKdef}. Notice that we can obtain $p_{\mathscr{C}}^{(r_1,\ldots,r_d)}$ by only doing function evaluations at points on $\Delta_{n_1}(r_1)\times \cdots\times\Delta_{n_d}(r_d)$.

\begin{theorem}\label{Th6}
For any given $r_1,\ldots,r_d\in \mathbb{N}$, one has
\begin{equation*}\label{approxrat1}
p_{\mathcal{A}}^{\rm min}-p^{(r_1,\ldots,r_d)}_{\mathscr{C}}\leq \frac{\tau(r_1,\ldots,r_d;d)}{\prod_{i=1}^d(r_i+1)}\left(p_{\mathcal{A}}^{\rm max}-p_{\mathcal{A}}^{\rm min}\right)
\end{equation*}
and
\begin{equation*}\label{approxrat2}
p^{(r_1,\ldots,r_d)}_{\Delta}-p_{\mathcal{A}}^{\rm min}\leq \frac{\tau(r_1,\ldots,r_d;d)}{\prod_{i=1}^d(r_i+2)}\left(p_{\mathcal{A}}^{\rm max}-p_{\mathcal{A}}^{\rm min}\right),
\end{equation*}
where $p_{\mathcal{A}}^{\rm max}$ is the maximum value of objective in \eqref{BQPS} and
$$
\begin{array}{lll}
\tau(r_1,\ldots,r_d;d)&:=&\underset{\nu=(\nu_1,\ldots,\nu_d)\in \{0,1\}^d,|\nu|=d-1}\sum(r_1+2)^{\nu_1}\cdots (r_d+2)^{\nu_d}\\
&&+\underset{\nu=(\nu_1,\ldots,\nu_d)\in \{0,1\}^d,|\nu|=d-3}\sum(r_1+2)^{\nu_1}\cdots (r_d+2)^{\nu_d}+\cdots.
\end{array}$$
\end{theorem}

\begin{proof}
By using similar arguments as for Theorem \ref{Th4}, we can obtain that
$$
p_{\mathscr{C}}^{(r_1,\ldots,r_d)}=\displaystyle\frac{\prod_{i=1}^d(r_i+2)}{\prod_{i=1}^d(r_i+1)}{\rm min}\left\{\tilde{p}_{\mathcal{A}}(x^{(1)},\ldots,x^{(d)})~|~x^{(k)}\in \Delta_{n_k}(r_k),~ k=1,\ldots,d\right\},
$$
where
$$
\begin{array}{l}
\tilde{p}_{\mathcal{A}}(x^{(1)},\ldots,x^{(d)})\\
=p_{\mathcal{A}}(x^{(1)},\ldots,x^{(d)})
~~\left.\displaystyle-\sum_{k=1}^d\sum_{i_k=1}^{n_k}x_{i_k}^{(k)}p_{\mathcal{A}^{i_k}_{k}}(x^{(1)},\ldots,x^{(k-1)},x^{(k+1)},\ldots,x^{(d)})\right.\\
~~+\displaystyle\sum_{1\leq k\neq s\leq d}\sum_{1\leq i_k\leq n_k,1\leq i_s\leq n_s}x^{(k)}_{i_{k}}x^{(s)}_{i_{s}}p_{\mathcal{A}_{ks}^{i_k,i_s}}(x^{(1)},\ldots,x^{(k-1)},x^{(k+1)},\ldots,x^{(s-1)},x^{(s+1)},\ldots,x^{(d)})\\
~~+\ldots\\
~~+(-1)^d\mathcal{A}_{1\ldots d}\bullet (x^{(1)}\otimes \cdots \otimes x^{(d)})\end{array}
$$
with $$\mathcal{A}^{i_k}_{k}=\displaystyle\frac{1}{r_k+2}(a_{i_1j_1\ldots i_{k-1}j_{k-1}i_ki_ki_{k+1}j_{k+1}\ldots i_{d}j_{d}})_{1\leq i_l,j_l\leq n_l,l\in [d]\backslash\{k\}},~~~~~~\forall~i_k=1,\ldots,n_k,$$
$$
\begin{array}{r}
\mathcal{A}_{ks}^{i_k,i_s}=\displaystyle\frac{1}{(r_k+2)(r_s+2)}(a_{i_1j_1\ldots i_{k-1}j_{k-1}i_ki_ki_{k+1}j_{k+1}\ldots i_{s-1}j_{s-1}i_{s}i_{s}i_{s+1}j_{s+1}\ldots i_{d}j_{d}})_{1\leq i_l,j_l\leq n_l,l\in [d]\backslash\{k,s\}},\\
~~~~~~~~~~~~~~~~~~~~~~~~~~~~~~~~~~~~~~~~~~~~~~~~~~~~~~~~~~~~~~~~~~~~~~~~~~~~~~~~~\forall~i_k=1,\ldots,n_k,~i_s=1,\ldots,n_s
,\end{array}
$$
$$
\ldots
$$
$$
\mathcal{A}_{1\ldots d}=\frac{1}{(r_1+2)\cdots(r_d+2)}(a_{i_1i_1i_2i_2\ldots i_di_d})_{1\leq i_k\leq n_k,k=1,\ldots,d}.
$$
Moreover, since $p_{\mathcal{A}}(x^{(1)},\ldots,x^{(d)})\geq p_{\mathcal{A}}^{\rm min}$, $p_{\mathcal{A}^{i_k}_{k}}(x^{(1)},\ldots,x^{(k-1)},x^{(k+1)},\ldots,x^{(d)})\leq p_{\mathcal{A}}^{\rm max}/(r_k+2)$, and so on, for any $x^{(k)}\in \Delta_{n_k}~(k=1,\ldots,d)$,  by using similar arguments as for Theorem \ref{Th5}, we have
\begin{equation*}\label{Mddrr}
  \begin{array}{lll}
p_{\mathscr{C}}^{(r_1,\ldots,r_d)}&\geq&\displaystyle\frac{\prod_{i=1}^d(r_i+2)}{\prod_{i=1}^d(r_i+1)}\left\{\left[1+\displaystyle\sum_{1\leq k\neq s\leq d}\frac{1}{(r_k+2)(r_s+2)}\right.\right.\\
&&\left.+\displaystyle\sum_{1\leq k\neq s\neq l\neq q\leq d}\frac{1}{(r_k+2)(r_s+2)(r_l+2)(r_q+2)}+\ldots\right]p_{\mathcal{A}}^{\rm min}\\
&&-\left.\left[\displaystyle\sum_{1\leq k\leq d}\frac{1}{(r_k+2)}+\sum_{1\leq k\neq s\neq l\leq d}\frac{1}{(r_k+2)(r_s+2)(r_l+2)}+\ldots\right]p_{\mathcal{A}}^{\rm max}\right\}.
\end{array}
\end{equation*}
Consequently, we obtain
\begin{equation}\label{DDEERRS}
 p_{\mathcal{A}}^{\rm min}-p_{\mathscr{C}}^{(r_1,\ldots,r_d)}\leq \frac{\tau(r_1,\ldots,r_d;d)}{\prod_{i=1}^d(r_i+1)}p_{\mathcal{A}}^{\rm max}-\frac{\bar\tau(r_1,\ldots,r_d;d)}{\prod_{i=1}^d(r_i+1)}p_{\mathcal{A}}^{\rm min},
\end{equation}
where $$
\begin{array}{lll}
\bar\tau(r_1,\ldots,r_d;d)&=&\displaystyle\prod_{i=1}^d(r_i+2)-\prod_{i=1}^d(r_i+1)+\sum_{l=(l_1,\ldots,l_d)\in \{0,1\}^d,|l|=d-2}(r_1+2)^{l_1}\cdots (r_d+2)^{l_d}\\
&&\displaystyle+\sum_{l=(l_1,\ldots,l_d)\in \{0,1\}^d,|l|=d-4}(r_1+2)^{l_1}\cdots (r_d+2)^{l_d}+\cdots.
\end{array}
$$
We claim that
\begin{equation}\label{tteeqq}
\bar\tau(r_1,\ldots,r_d;d)=\tau(r_1,\ldots,r_d;d).
\end{equation}
In fact, since $\prod_{i=1}^d(x-a_i)=x^d-(\sum_{i=1}^da_i)x^{d-1}+\cdots+(-1)^d\prod_{i=1}^da_i$, by taking $x=1$ and $a_i=r_i+2$, it holds that
$$\prod_{i=1}^d(r_i+1)=(-1)^d+(-)^{d-1}(\sum_{i=1}^d(r_i+2)x^{d-1}+(-1)^{d-2}\sum_{1\leq i\neq j\leq d}(r_i+2)(r_j+2)+\cdots+\prod_{i=1}^d(r_i+2),$$ which implies
$$\begin{array}{l}
\displaystyle\left[\prod_{i=1}^d(r_i+2)+\sum_{l=(l_1,\ldots,l_d)\in \{0,1\}^d,|l|=d-2}(r_1+2)^{l_1}\cdots (r_d+2)^{l_d}\right.\\
~~~~~~~~~~~~~\displaystyle\left.+\sum_{l=(l_1,\ldots,l_d)\in \{0,1\}^d,|l|=d-4}(r_1+2)^{l_1}\cdots (r_d+2)^{l_d}+\cdots\right]-\prod_{i=1}^d(r_i+1)\\
=\displaystyle\sum_{l=(l_1,\ldots,l_d)\in \{0,1\}^d,|l|=d-1}(r_1+2)^{l_1}\cdots (r_d+2)^{l_d}\\
~~~~~~~~~~~~~\displaystyle+\sum_{l=(l_1,\ldots,l_d)\in \{0,1\}^d,|l|=d-3}(r_1+2)^{l_1}\cdots (r_d+2)^{l_d}+\ldots,
\end{array}
$$
and hence (\ref{tteeqq}) holds. By (\ref{DDEERRS}) and (\ref{tteeqq}), we obtain the first inequality. The second inequality can be proved by a similar way.
\qed\end{proof}

From Theorems \ref{Th5} and \ref{Th6}, we know that a PTAS for solving (\ref{polyOpt}) exists and it can be extended to solving (\ref{BQPS}). Specially, for (\ref{BQPS}), we know, from Theorem \ref{Th6}, that if $r_1=\ldots =r_d=r$, then one has
\begin{equation*}\label{approxrat10}
p_{\mathcal{A}}^{\rm min}-p^{(r,\ldots,r)}_{\mathscr{C}}\leq \frac{\tilde\tau(r;d)}{(r+1)^d}\left(p_{\mathcal{A}}^{\rm max}-p_{\mathcal{A}}^{\rm min}\right)
\end{equation*}
and
\begin{equation*}\label{approxrat20}
p^{(r,\ldots,r)}_{\Delta}-p_{\mathcal{A}}^{\rm min}\leq \frac{\tilde \tau(r;d)}{(r+2)^d}\left(p_{\mathcal{A}}^{\rm max}-p_{\mathcal{A}}^{\rm min}\right),
\end{equation*}
where $\tilde\tau(r;d):=\binom{d}{1}(r+2)^{d-1}+\binom{d}{3}(r+2)^{d-3}+\cdots$. In this case, we have
\begin{equation*}\label{approxrat101}
p_{\mathcal{A}}^{\rm min}-p^{(r,\ldots,r)}_{\mathscr{C}}\leq O(1/(r+1))\left(p_{\mathcal{A}}^{\rm max}-p_{\mathcal{A}}^{\rm min}\right)
\end{equation*}
and
\begin{equation*}\label{approxrat201}
p^{(r,\ldots,r)}_{\Delta}-p_{\mathcal{A}}^{\rm min}\leq O(1/(r+2))\left(p_{\mathcal{A}}^{\rm max}-p_{\mathcal{A}}^{\rm min}\right).
\end{equation*}

\section{Numerical illustration}

In this section, we provide some preliminary numerical results to show that our approximation approach performs reliably on approximately solving StBQP. All codes were written by {\sc Matlab} 2010b and all the numerical tests were conducted on a Lenovo desktop computer with Intel Pentium Dual-Core processor 2.33GHz and 2GB main memory.

\begin{example}\label{exam2}
This example comes from \cite[Example 2]{BLQZ12}, where the objective function defined in \eqref{polyOpt} is specified as
\begin{align*}
p_{\A}(x,y):= \sum_{j=1}^4 a_j(x_1^2+x_2^2)y_j^2 + \sum_{j=1}^3 4b_jx_1x_2y_jy_{j+1}
\end{align*}
and $(x,y)\in\Delta_2\times\Delta_4$, where the vectors $a$ and $b$ are respectively defined by $$a:=[0.7027,0.1536,0.9535,0.5409]^\t\qquad \text{and}\qquad b:=[1.6797,1.0366,1.8092]^\t.$$
As shown in \cite{BLQZ12}, the optimal objective value of this problem is $0.0598$, that is, $p_{\A}^{\min}=0.0598$.
\end{example}

\begin{example}\label{exam1}
The second one is taken from \cite[Example 1]{LNQY09}, which takes the form of
\begin{align*}
p_{\A}(x,y):= x_1^2y_1^2+x_2^2y_2^2+x_3^2y_3^2+2(x_1^2y_2^2+x_2^2y_3^2+x_3^2y_1^2)
-2(x_1x_2y_1y_2+x_1x_3y_1y_3+x_2x_3y_2y_3)
\end{align*}
and $(x,y)\in\Delta_3\times\Delta_3.$ In \cite{BLQZ12} (see also \cite{Cho75}), the authors showed that the global optimal value of the objective function is $0$, that is, $p_{\A}^{\min}=0$.
\end{example}

In the following three examples, we construct the objective function defined in \eqref{polyOpt} in the form of
\begin{equation}\label{Exam}
p_{\A}(x,y) = \A\bullet {\mathcal Z},
\end{equation}
where $\A$ is a $4$-th order tensor and $\mathcal{Z}:=(xx^\top)\otimes (yy^\top)$.


\begin{example}\label{exam3}
The corresponding tensor defined in \eqref{Exam} is given by $\mathcal{A}:=A\otimes B-2C\otimes D$, where
$$
A=\left(
\begin{array}{ccc}
1&2&1\\
2&4&2\\
1&2&1
\end{array}
\right),\; ~B=\left(
\begin{array}{ccc}
1&1&2\\
1&1&2\\
2&2&4
\end{array}
\right),\;~C=\frac{1}{2}\left(
\begin{array}{ccc}
2&3&2\\
3&4&3\\
2&3&2
\end{array}
\right)\;~{\rm and~}D=\frac{1}{2}\left(
\begin{array}{ccc}
2&2&3\\
2&2&3\\
3&3&4
\end{array}
\right).
$$
The simplex constraint is $(x,y)\in \Delta_3\times\Delta_3$. We can verify that the corresponding optimal value is $p^{\min}_\A=-1$.
\end{example}

\begin{example}\label{exam4}
We construct the tensor $\A$ defined in \eqref{Exam} by $\mathcal{A}:=A\otimes B-C\otimes D$ with
$$
A=\left(
\begin{array}{cccc}
1&-3&-2&-1\\
-3&9&6&3\\
-2&6&4&2\\
-1&3&2&1
\end{array}
\right), ~B=\left(
\begin{array}{ccccc}
4&-4&-2&-2&-2\\
-4&4&2&2&2\\
-2&2&1&1&1\\
-2&2&1&1&1\\
-2&2&1&1&1
\end{array}
\right),~C=\left(
\begin{array}{cccc}
-2&2&1&0\\
2&6&5&4\\
1&5&4&3\\
0&4&3&2
\end{array}
\right),~{\rm and}~
D=\left(
\begin{array}{ccccc}
-4&0&-1&-1&-1\\
0&4&3&3&3\\
-1&3&2&2&2\\
-1&3&2&2&2\\
-1&3&2&2&2
\end{array}
\right).
$$
The simplex constraint corresponds to $(x,y)\in\Delta_4\times\Delta_5$, and the global optimal value is $p^{\min}_\A=-4$.
\end{example}

\begin{example}\label{exam5}
For any given positive integer $l$, let $e^{(l)}$ be the $l$-dimensional vector of ones, that is, $e^{(l)}:=(1,1,\cdots,1)^\top$. Correspondingly, let $E_{l}$ be an $l\times l$ matrix of ones. We generate the $4$-th order tensor $\A$ defined in \eqref{Exam} by $\mathcal{A}:=A\otimes B-2C\otimes D$, here
$$
A=\left(
\begin{array}{ccc}
1&-2&-(e^{(n-2)})^\top\\
-2&4&2(e^{(n-2)})^\top\\
-e^{(n-2)}&2e^{(n-2)}&E_{n-2}
\end{array}
\right), ~B=\left(
\begin{array}{ccccc}
1&-1&-2&-(e^{(m-3)})^\top\\
-1&1&2&(e^{(m-3)})^\top\\
-2&2&4&2(e^{(m-3)})^\top\\
-e^{(m-3)}&e^{(m-3)}&2e^{(m-3)}&E_{m-3}
\end{array}
\right)$$
and $$~C=\left(
\begin{array}{ccc}
-1&1/2&0\\
1/2&2&(3/2)(e^{(n-2)})^\top\\
0&(3/2)e^{(n-2)}&E_{n-2}
\end{array}
\right),~
D=\left(
\begin{array}{ccccc}
-1&0&1/2&0\\
0&1&3/2&(e^{(m-3)})^\top\\
1/2&3/2&2&(3/2)(e^{(m-3)})^\top\\
0&e^{(m-3)}&(3/2)e^{(m-3)}&E_{m-3}
\end{array}
\right).
$$
In addition, the simplex constraint corresponds to $(x,y)\in\Delta_n\times\Delta_m$. For any $n\geq 2$ and $m\geq 3$, similarly, we can verify that such problem attains its optimal value $p^{\min}_\A=-1$. In our test, we focus on the case $n=5$ and $m=8$.
\end{example}

Note that our approximation approach is closely related to the choices of $s$ and $r$. Actually, it can be easily seen from Theorem \ref{Th5} that larger $s$ and $r$ could lead to better approximate results. Therefore, we here investigate the behaviors of $s$ and $r$ on approximately solving the problem under consideration. More concretely, we test $12$ groups of $(s,r)$ in the combinations of $s:=\{3,4,8,13\}$ and $r:=\{5,12,17\}$. The corresponding numerical results are summarized in Table \ref{table1}, where $p_{\Delta}^{(s,r)}$ and $p_{\A}^{\min}$ represent the approximate optimal value and the exact optimal value, respectively. As a matter of fact, our examples might have many optimal solutions. Thus, we show graphically the corresponding approximate optimal solutions in Figs. \ref{fig1} and \ref{fig2}.

\begin{table}[!htb]
\begin{center}
\caption{Comparison between approximate optimal values and exact optimal values.}\vskip 0.2mm
\label{table1}
\def\temptablewidth{1\textwidth}
\begin{tabular*}{\temptablewidth}{@{\extracolsep{\fill}}clcccccccccccccc}\toprule
&&\multicolumn{2}{c}{Example \ref{exam2}} &&\multicolumn{2}{c}{Example \ref{exam1}}&&\multicolumn{2}{c}{Example \ref{exam3}}&&\multicolumn{2}{c}{Example \ref{exam4}}&&\multicolumn{2}{c}{Example \ref{exam5}}\\
\cline{3-4} \cline{6-7}\cline{9-10}\cline{12-13}\cline{15-16}
$i$ & $(s,r)$& $p_{\Delta}^{(s,r)}$& $p_{\cal A}^{\min}$ && $p_{\Delta}^{(s,r)}$& $p_{\cal A}^{\min}$ && $p_{\Delta}^{(s,r)}$& $p_{\cal A}^{\min}$ && $p_{\Delta}^{(s,r)}$& $p_{\cal A}^{\min}$ && $p_{\Delta}^{(s,r)}$& $p_{\cal A}^{\min}$\\\midrule
1 &(3,5)& 0.0666 & 0.0598 && 0.00 & 0.00 && -1.00 & -1.00 && -4.00 & -4.00 && -1.00 & -1.00 \\
2 &(3,12)& 0.0668 & 0.0598 && 0.00 & 0.00 && -1.00 & -1.00 && -4.00 & -4.00 && -1.00 & -1.00 \\
3 &(3,17)& 0.0665 & 0.0598 && 0.00 & 0.00 && -1.00 & -1.00 && -4.00 & -4.00 && -1.00 & -1.00 \\
4 &(4,5) & 0.0600 & 0.0598 && 0.00 & 0.00 && -1.00 & -1.00 && -4.00 & -4.00 && -1.00 & -1.00 \\
5 &(4,12)& 0.0601 & 0.0598 && 0.00 & 0.00 && -1.00 & -1.00 && -4.00 & -4.00 && -1.00 & -1.00 \\
6 &(4,17)& \textbf{0.0599} & 0.0598 && 0.00 & 0.00 && -1.00 & -1.00 && -4.00 & -4.00 && -1.00 & -1.00 \\
7 &(8,5)& 0.0600 & 0.0598 && 0.00 & 0.00 && -1.00 & -1.00 && -4.00 & -4.00 && -1.00 & -1.00 \\
8 &(8,12)& 0.0601 & 0.0598 && 0.00 & 0.00 && -1.00 & -1.00 && -4.00 & -4.00 && -1.00 & -1.00 \\
9 &(8,17)& \textbf{0.0599} & 0.0598 && 0.00 & 0.00 && -1.00 & -1.00 && -4.00 & -4.00 && -1.00 & -1.00 \\
10&(13,5)& 0.0603 & 0.0598 && 0.00 & 0.00 && -1.00 & -1.00 && -4.00 & -4.00 && -1.00 & -1.00 \\
11&(13,12)& 0.0605& 0.0598&&0.00 & 0.00 && -1.00 & -1.00 && -4.00 & -4.00 && -1.00 & -1.00 \\
12&(13,17)& 0.0602& 0.0598&&0.00 & 0.00&& -1.00 & -1.00&& -4.00& -4.00&& -1.00& -1.00 \\
\bottomrule
\end{tabular*}
\end{center}
\end{table}

\begin{figure}[!htb]
\includegraphics[width=0.48\textwidth]{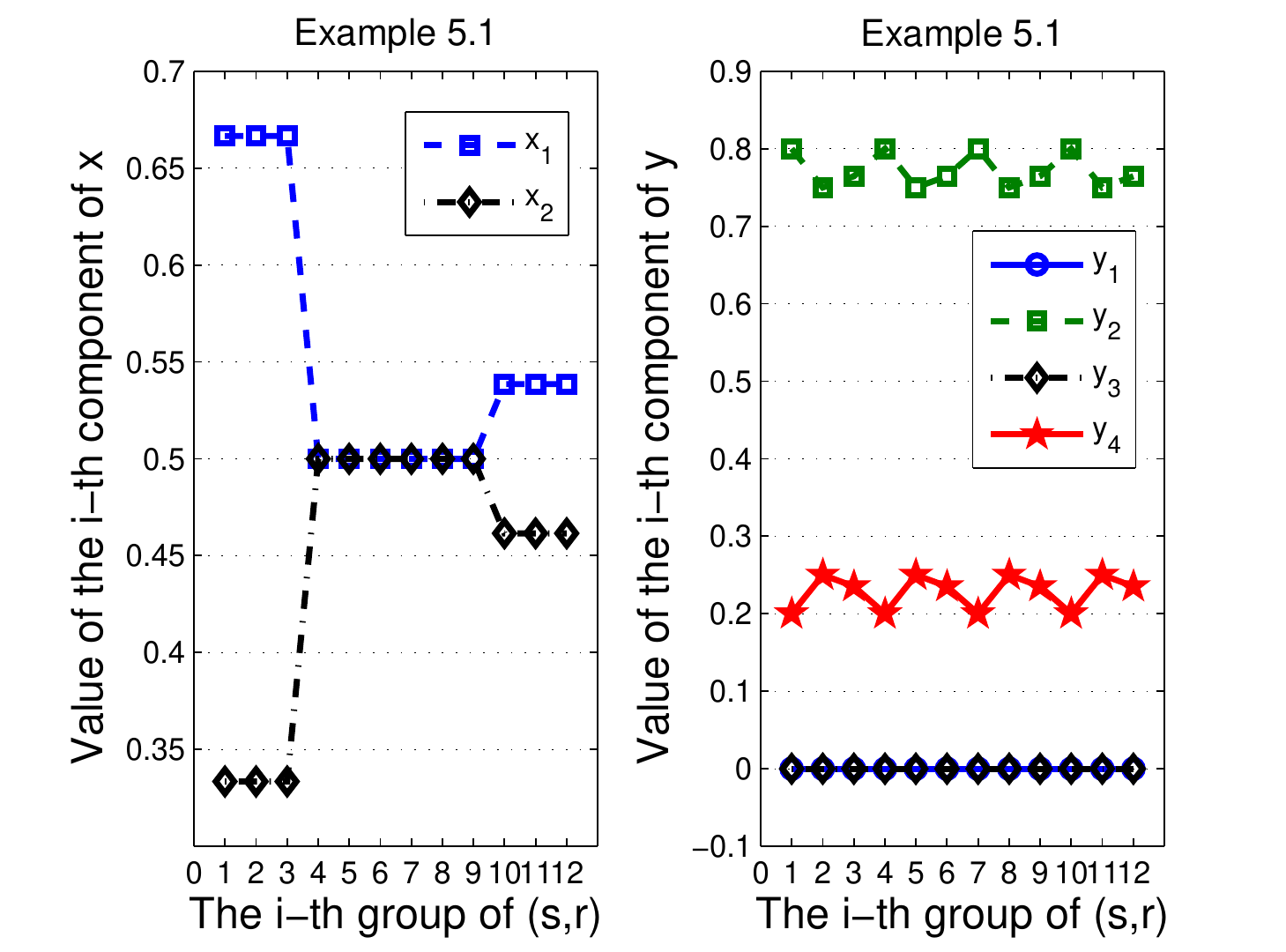}
\includegraphics[width=0.48\textwidth]{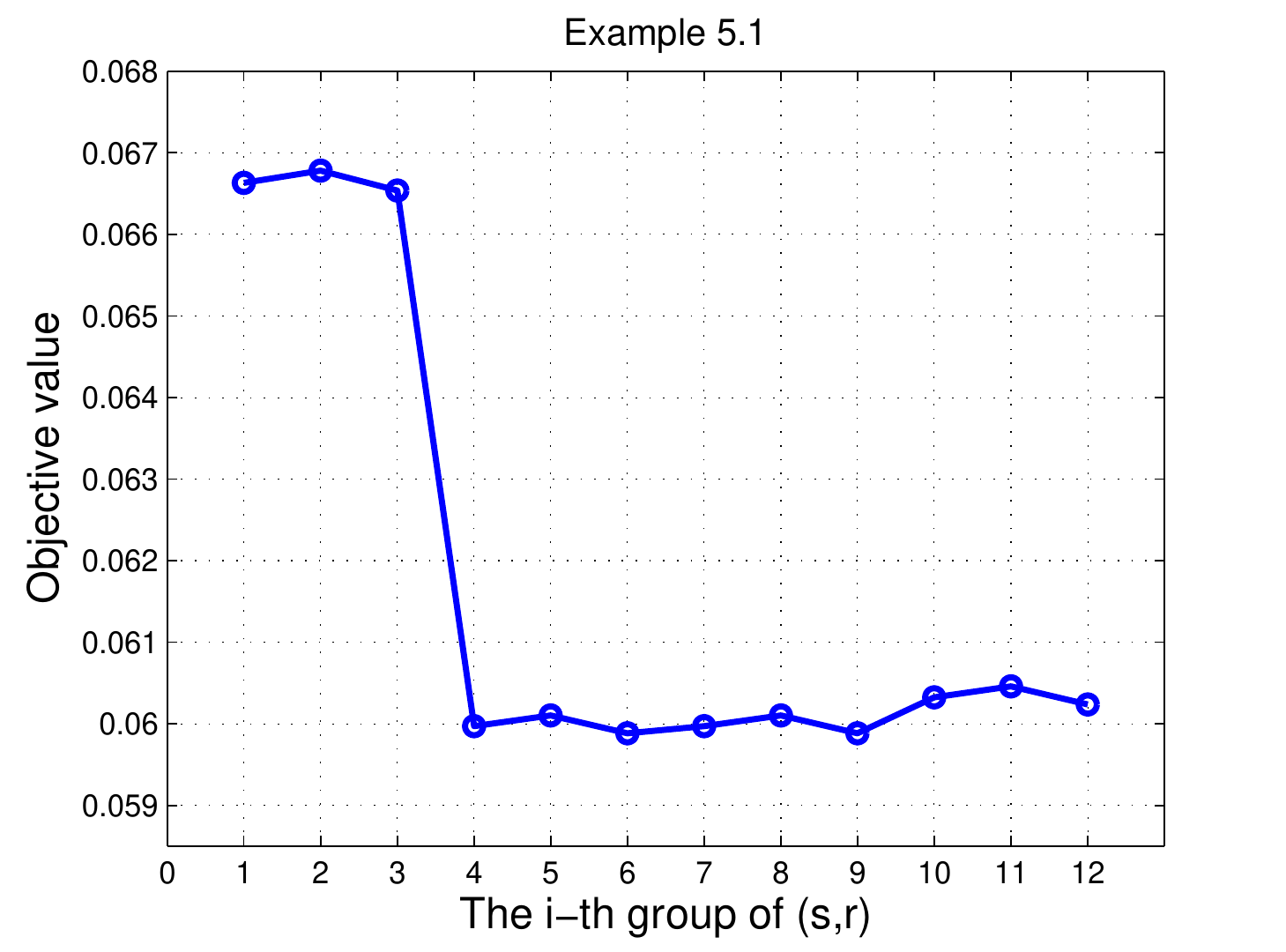}
\caption{Approximate optimal solutions and objective value of Example \ref{exam2} with respect to ($s,r$).}\label{fig1}
\end{figure}

\begin{figure}[!htb]
\includegraphics[width=0.48\textwidth]{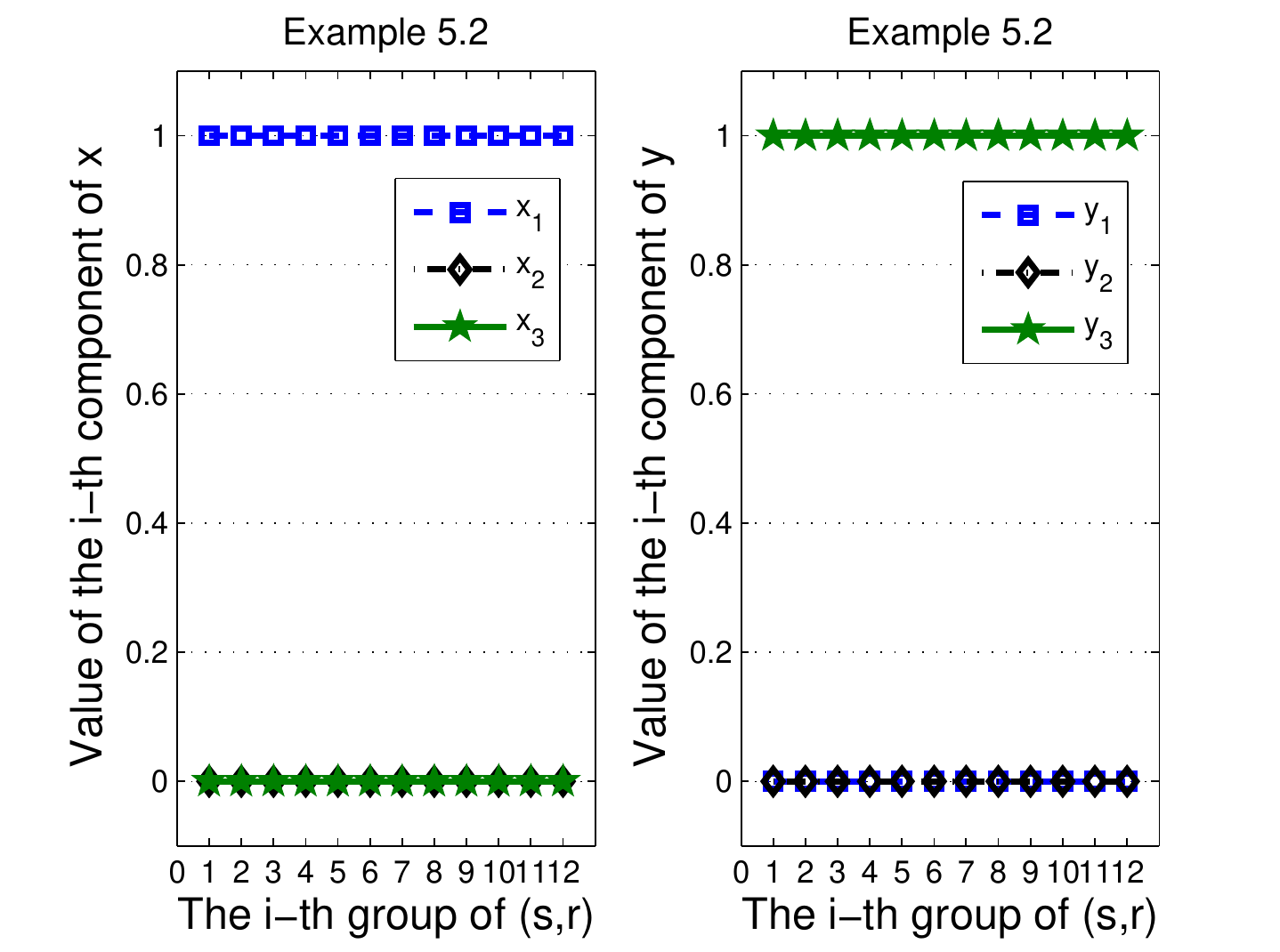}
\includegraphics[width=0.48\textwidth]{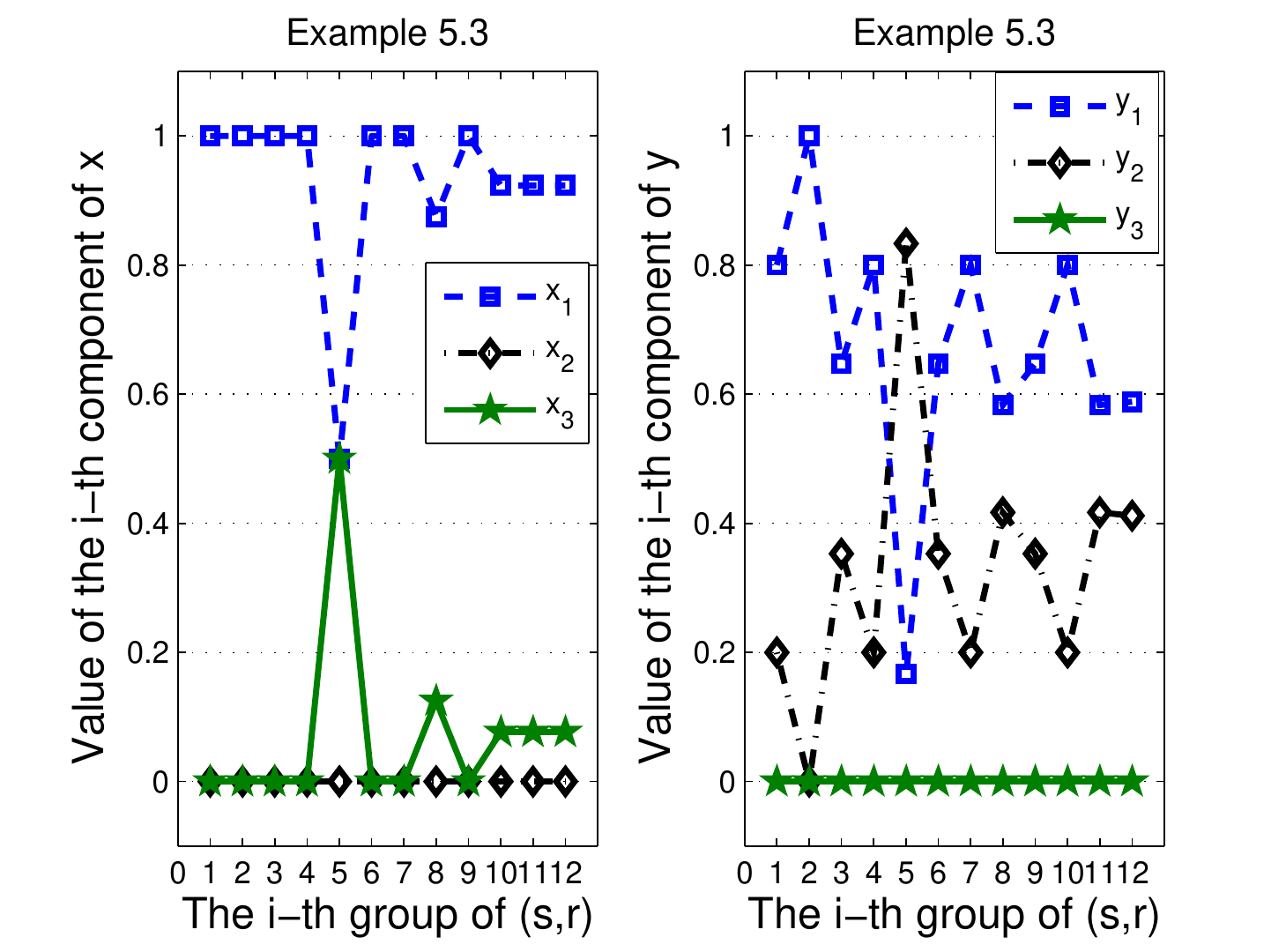}\\
\includegraphics[width=0.48\textwidth]{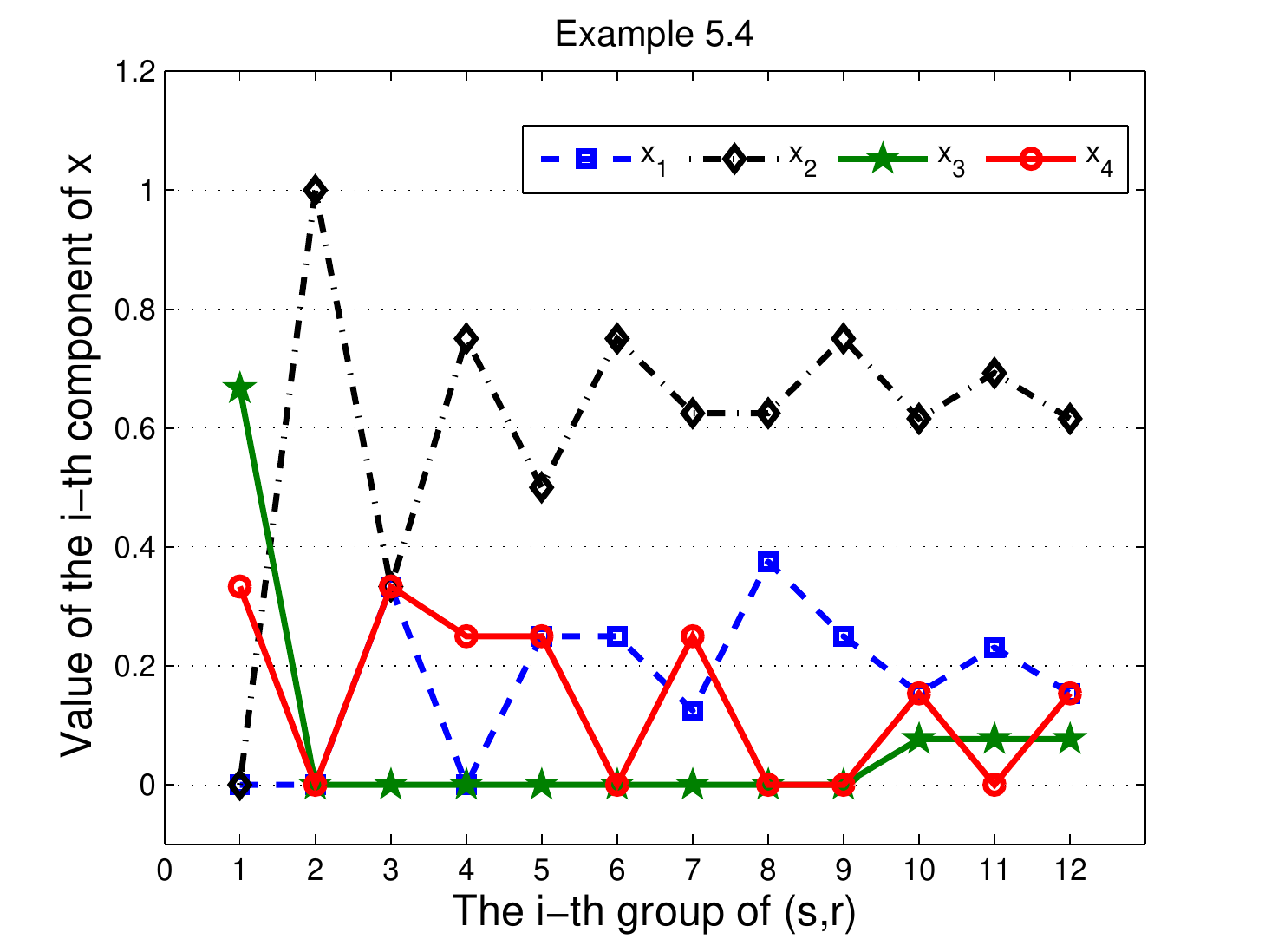}
\includegraphics[width=0.48\textwidth]{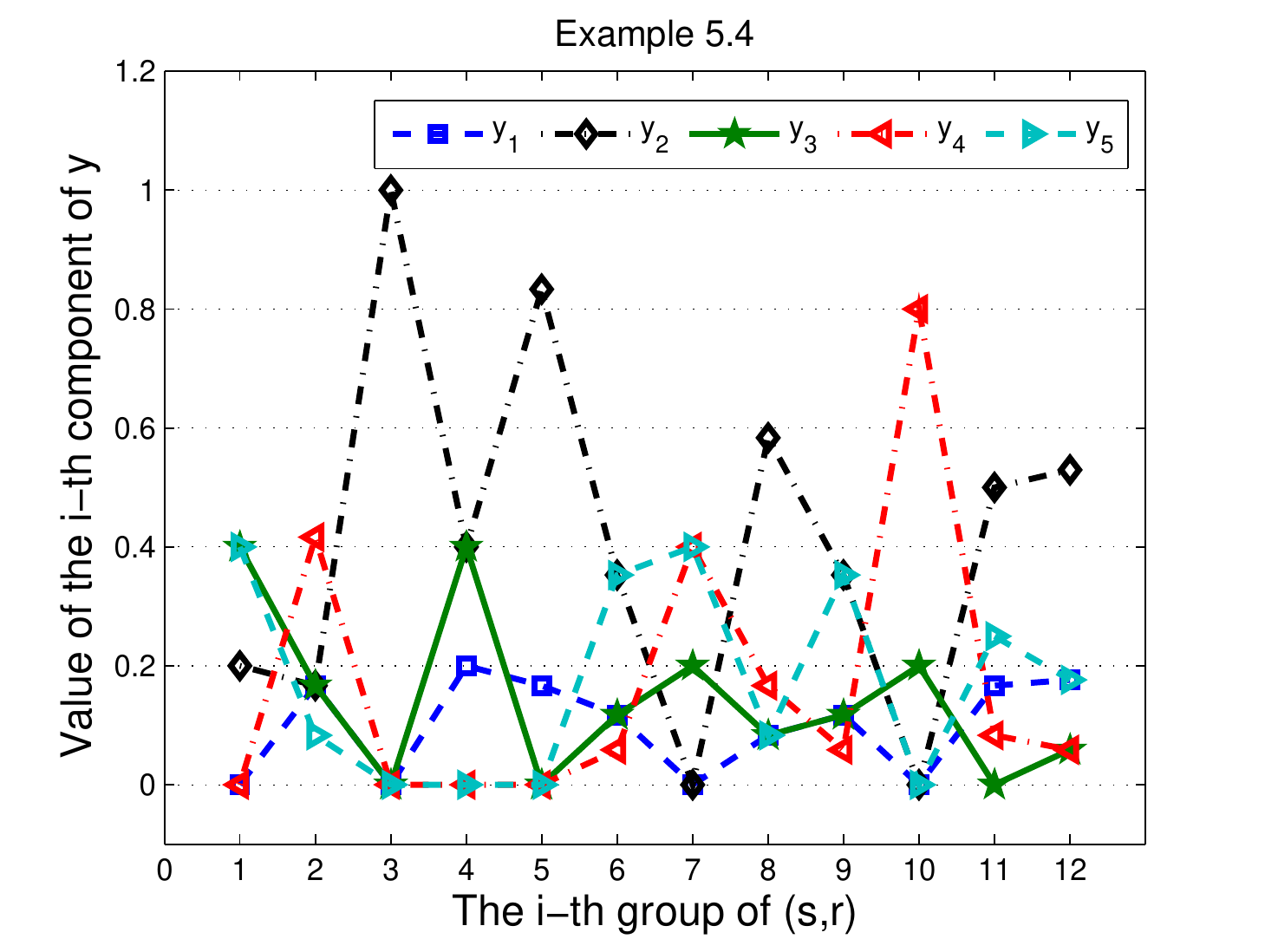}\\
\includegraphics[width=0.48\textwidth]{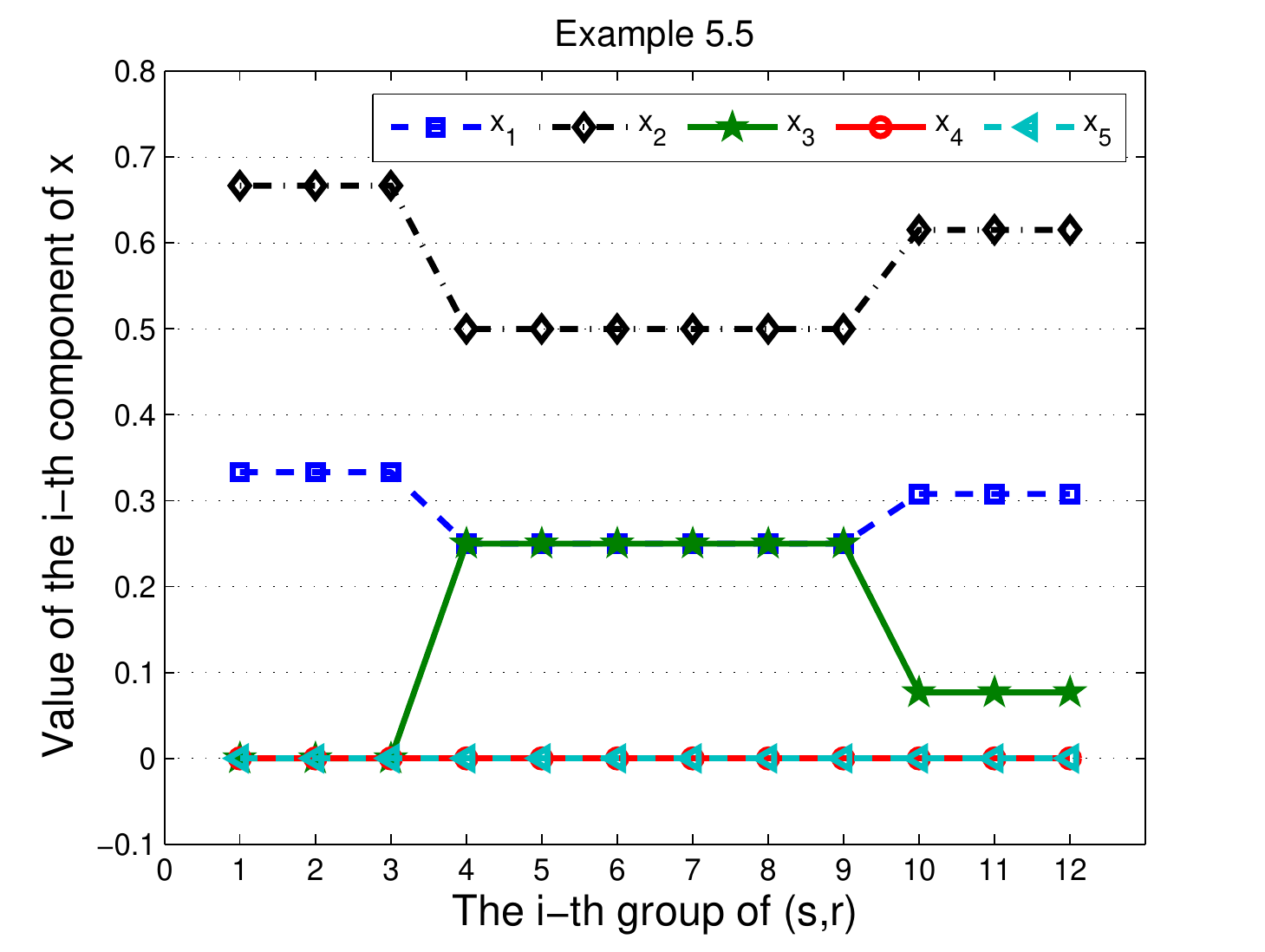}
\includegraphics[width=0.48\textwidth]{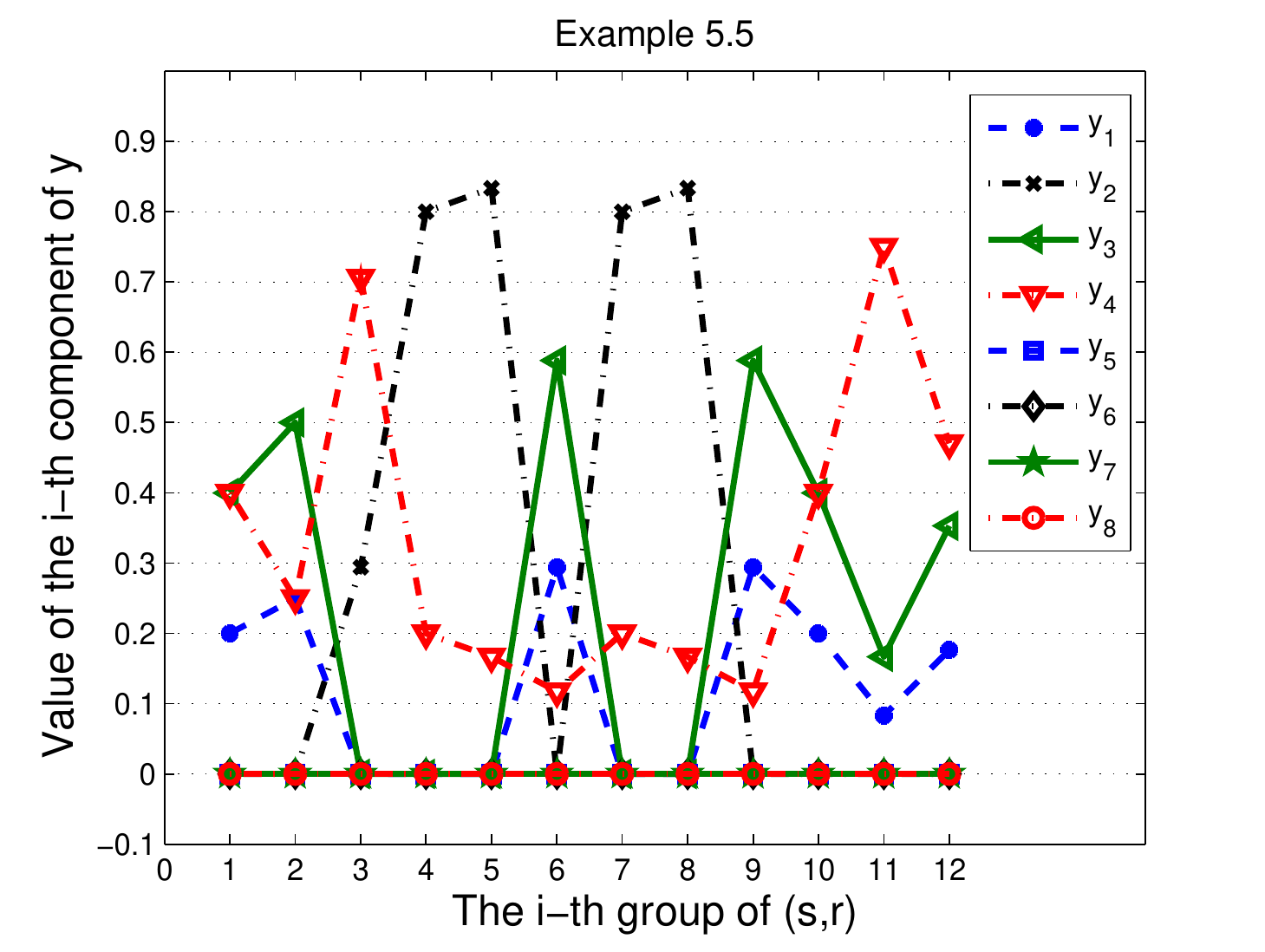}
\caption{Approximate optimal solutions of Examples \ref{exam1}--\ref{exam5} with respect to ($s,r$).}\label{fig2}
\end{figure}

It is clear from Table \ref{table1} and Fig. \ref{fig1} that $s=4$ and $s=8$ yield the best approximate results for Example \ref{exam2}. Moreover, for a fixed $s$, we can easily see that larger $r$ leads to better approximation, which also supports our theoretical results in Theorem \ref{Th5}. More surprisingly, the results in Table \ref{table1} corresponding to Examples \ref{exam1}--\ref{exam5} show that our approximation strategy can exactly get the accurate optimal values in addition to finding different optimal solutions (see Fig. \ref{fig2}). Thus, our results further verify the reliability of the proposed approach.

\section{Conclusions}
In this paper, we suggested an approach to approximating the optimal value of StBQP. After reformulating the original  problem as an equivalent copositive tensor programming problem, a quality of approximation was presented, which is based upon the approximation of the reformulated cone of copositive tensors by a serial polyhedral cones. The obtained quality of approximation showed that a PTAS for StBQP exists, and extended the previously best known approximation result on StQP due to Bomze et al \cite{BK02} to StBQP version. Finally, a quality of approximation for StMQP was also presented.

\begin{acknowledgements}
The first two authors were supported by National Natural Science Foundation of China (NSFC) at Grant Nos. (11171083, 11301123) and the Zhejiang Provincial NSFC at Grant No. LZ14A010003. The third author was supported by the Hong Kong Research Grant Council (Grant Nos. PolyU 502510, 502111, 501212 and 501913).
\end{acknowledgements}


\begin{thebibliography}{10}
\providecommand{\url}[1]{{#1}}
\providecommand{\urlprefix}{URL }
\expandafter\ifx\csname urlstyle\endcsname\relax
  \providecommand{\doi}[1]{DOI~\discretionary{}{}{}#1}\else
  \providecommand{\doi}{DOI~\discretionary{}{}{}\begingroup
  \urlstyle{rm}\Url}\fi

\bibitem{BBPP99}
Bomze, I.M., Budinich, M., Pardalos, P.M., Pelillo, M.: The maximum clique
  problem.
\newblock In: D.Z. Du, P.~Pardalos (eds.) Handbook of Combinatorial
  Optimization, pp. 1--74. Springer, Kluwer, Dordrecht (1999)

\bibitem{BK02}
Bomze, I.M., De~Klerk, E.: Solving standard quadratic optimization problems via
  linear, semidefinite and copositive programming.
\newblock J. Global Optim. \textbf{24}, 163--185 (2002)

\bibitem{BGY14}
Bomze, I.M., Gollowitzer, S., Y{\i}ld{\i}r{\i}m, E.A.: Rounding on the standard
  simplex: {R}egular grids for global optimization.
\newblock J. Global Optim. \textbf{59}, 243--258 (2014)

\bibitem{BLQZ12}
Bomze, I.M., Ling, C., Qi, L., Zhang, X.: Standard bi-quadratic optimization
  problems and unconstrained polynomial reformulations.
\newblock J. Global Optim. \textbf{52}, 663--687 (2012)

\bibitem{BP05}
Bomze, I.M., Palagi, L.: Quartic formulation of standard quadratic optimization
  problems.
\newblock J. Global Optim. \textbf{32}, 181--205 (2005)

\bibitem{Bos83}
Bos, L.P.: Bounding the {L}ebesgue function for lagrange interpolation in a
  simplex.
\newblock J. Approx. Theory \textbf{38}, 43--59 (1983)

\bibitem{Cho75}
Choi, M.D.: Positive semidefinite biquadratic forms.
\newblock Linear Alg. Appl. \textbf{12}, 95--100 (1975)

\bibitem{KLP06}
De~Klerk, E., Laurent, M., Parrilo, P.A.: A {PTAS} for the minimization of
  polynomials of fixed degree over the simplex.
\newblock Theor. Comput. Sci. \textbf{361}, 210--225 (2006)

\bibitem{KLS14}
De~Klerk, E., Laurent, M., Sun, Z.: An alternative proof of a {PTAS} for
  fixed-degree polynomial optimization over the simplex.
\newblock Math. Program. \doi{10.1007/s10107-014-0825-6}

\bibitem{HLP52}
Hardy, G.H., Littlewood, J.E., P{\'o}lya, G.: Inequalities, 2nd edn.
\newblock Cambridge University Press (1952)

\bibitem{LHZ12}
Li, Z., He, S., Zhang, S.: Approximation Methods for Polynomial Optimization
  Methods, Algorithms, and Applications.
\newblock Springer, New York (2012)

\bibitem{LNQY09}
Ling, C., Nie, J., Qi, L., Ye, Y.: Biquadratic optimization over unit spheres
  and semidefinite programming relaxations.
\newblock SIAM J. Optim. \textbf{20}, 1286--1310 (2009)

\bibitem{LZQ14}
Ling, C., Zhang, X.Z., Qi, L.: Approximation bound analysis for the standard
  multi-quadratic optimization problem.
\newblock Tech. rep., School of Science, Hangzhou Dianzi University (2014)

\bibitem{M52}
Markowitz, H.: Portfolio selection.
\newblock J. Financ. \textbf{7}, 77--91 (1952)

\bibitem{MK87}
Murty, K.G., Kabadi, S.N.: Some {NP}-complete problems in quadratic and
  nonlinear programming.
\newblock Math. Program. \textbf{39}, 117--129 (1987)

\bibitem{Nes99}
Nesterov, Y.: Global quadratic optimization on the sets with simplex structure.
\newblock Tech. rep., Katholic University of Louvain, Belgium (1999)

\bibitem{NWY00}
Nesterov, Y., Wolkowicz, H., Ye, Y.: Nonconvex quadratic optimization.
\newblock In: H.~Wolkowicz, R.~Saigal, L.~Vandenberghe (eds.) Handbook of
  Semidefinite Programming: Theory, Algorithms, and Applications, pp. 361--416.
  Kluwer Acadamic Publishers, Dordrecht (2000)

\bibitem{P80}
Pang, J.S.: A new and efficient algorithm for a class of portfolio selection
  problems.
\newblock Oper. Res. \textbf{28}, 754--767 (1980)

\bibitem{Poly74}
P\'{o}lya, G.: \"{U}ber positive darstelluny von polynomen vierteljschr.
\newblock In: Naturforsch. Ges. Zurich, vol.~73, pp. 141--145. MIT Press (1974)

\bibitem{PR01}
Powers, V., Reznick, B.: A new bound for {P}{\'o}lya's theorem with
  applications to polynomials positive on polyhedra.
\newblock J. Pure Appl. Algebra \textbf{164}, 221--229 (2001)

\bibitem{Q13}
Qi, L.: Symmetric nonnegative tensors and copositive tensors.
\newblock Linear Alg. Appl. \textbf{439}, 228--238 (2013)

\bibitem{R01}
Renegar, J.: A mathematical view of interior-point methods in convex
  optimization.
\newblock MPS/SIAM Series on Optimization. SIAM, Philadelphia, PA (2001)

\bibitem{SY13}
Sagol, G., Y{\i}ld{\i}r{\i}m, E.A.: Analysis of copositive optimization based
  bounds on standard quadratic optimization.
\newblock Tech. rep., Department of Industrial Engineering, Koc University,
  Sariyer, Istanbul, Turkey (2013)

\bibitem{So11}
So, A.M.C.: Deterministic approximation algorithms for sphere constrained
  homogeneous polynomial optimization problems.
\newblock Math. Program. \textbf{129}, 357--382 (2011)

\bibitem{SQ14}
Song, Y., Qi, L.: Necessary and sufficient conditions for copositive tensors.
\newblock Linear Multilinear Algebra pp. 1--12.
\newblock \doi{10.1080/03081087.2013.851198}

\bibitem{Y12}
Y{\i}ld{\i}r{\i}m, E.A.: On the accuracy of uniform polyhedral approximations
  of the copositive cone.
\newblock Optim. Method Softw. \textbf{27}, 155--173 (2012)

\end{thebibliography}

\end{document}